\documentclass[11pt]{article}
\usepackage{graphicx}
\usepackage{graphics}
\usepackage{amsfonts}
\usepackage{amscd}
\usepackage{amsthm}
\usepackage{amsmath}
\usepackage{indentfirst}
\usepackage[all]{xy}
\usepackage{amssymb, amsmath, amsthm, amsgen, amstext, amsbsy, amsopn}
\usepackage{epic,eepic}
\usepackage{hyperref}
\usepackage{enumitem}
\usepackage{color}
\usepackage{dsfont}
\usepackage{comment}
\usepackage[margin=1.3in]{geometry}

\usepackage{eufrak}

\usepackage{appendix}

\theoremstyle{plain}
\newtheorem*{thm*}{Theorem}
\newtheorem*{thmA}{Theorem A}
\newtheorem*{thmB}{Theorem B}
\newtheorem*{thmC}{Theorem C}
\newtheorem*{thmC'}{Theorem C'}
\newtheorem{thm}{Theorem}
\newtheorem{lemma}{Lemma}[section]
\newtheorem{prop}[lemma]{Proposition}
\newtheorem{claim}[lemma]{Claim}
\newtheorem{cor}[lemma]{Corollary}

\newcommand{\vertiii}[1]{{\left\vert\kern-0.25ex\left\vert\kern-0.25ex\left\vert #1 
    \right\vert\kern-0.25ex\right\vert\kern-0.25ex\right\vert}}

\theoremstyle{definition}
\newtheorem{defn}[lemma]{Definition}
\newtheorem{question}[lemma]{Question}

\title{\textbf{\boldmath Flexibility of the Hamiltonian adjoint action and classification of bi-invariant metrics}}

\author{\textbf{Lev Buhovsky, Maksim Stoki\'c}}

\begin{document}

\maketitle

\begin{abstract}
    On an open, connected symplectic manifold $(M,\omega)$, the group of Hamiltonian diffeomorphisms forms an infinite-dimensional Fr\'echet Lie group with Lie algebra $C^{\infty}_c(M)$ and adjoint action given by pullbacks. We prove that this action is flexible: for any non-constant $u \in C^{\infty}(M)$, every $f \in C^{\infty}_c(M)$ can be expressed as a weighted finite sum of elements from the adjoint orbit of $u$, with total weight bounded by constant multiple of $\|f\|_{\infty} + \|f\|_{L^1}$. Consequently, all $\mathrm{Ham}(M,\omega)$-invariant norms on $C^{\infty}_c(M)$ are dominated by a sum of $L^{\infty}$ and $L^1$ norms. As an application, we classify up to equivalence all bi-invariant pseudo-metrics on the group of Hamiltonian diffeomorphisms of an exact symplectic manifold, answering a question of Eliashberg and Polterovich.
\end{abstract}

\tableofcontents

\section{Introduction}

The group of Hamiltonian diffeomorphisms $\mathrm{Ham}(M,\omega)$, of a symplectic manifold $(M,\omega)$, is an infinite dimensional Fr\'echet Lie group, whose Lie algebra $\mathcal{A}$ is isomorphic to
\begin{itemize}
    \item the space of zero-mean normalized functions
    \[C^{\infty}_0(M):=\big\{f\in C^{\infty}(M)\Big\vert \int_{M}f\omega^n=0\big\}\]
    if the manifold $M$ is closed,
    \item the space of compactly supported functions $C^{\infty}_c(M)$, if the manifold $M$ is open.
\end{itemize}
In either case, the \textit{adjoint action} of $\mathrm{Ham}(M,\omega)$ on the Lie algebra $\mathcal{A}$ is given by pull-backs: 
\[\varphi\in\mathrm{Ham}(M,\omega),\,f\in \mathcal{A},\quad\quad\mathrm{Ad}_{\varphi}\,f=f\circ\varphi^{-1}.\]

\noindent If the manifold $M$ is closed and connected, we established in our earlier work that the adjoint action satisfies the following flexibility property:

\begin{thmA}[Theorem 1 in Buhovsky, Stoki\'c \cite{BS25}]\label{TheoremA}
    Let $(M,\omega)$ be a closed and connected symplectic manifold, and let $u\in C^{\infty}_{0}(M)$ be a non-zero function. There exists $N=N(u)\in\mathbb{N}$ such that for any $f\in C^{\infty}_0(M)$ with $\|f\|_{\infty}\leq 1$, one can write
    \[f=\sum_{i=1}^N\Phi_i^*u,\]
    for some Hamiltonian diffeomorphisms $\Phi_i\in\mathrm{Ham}(M,\omega)$.
\end{thmA}

\noindent An important corollary of the Theorem \hyperref[TheoremA]{A} is that any $\mathrm{Ham}(M,\omega)$-invariant norm $\|\cdot\|$ on the space $C^{\infty}_0(M)$ satisfies $\|\cdot\|\leq C\cdot \|\cdot\|_{\infty}$ for some constant $C$, which has further implications in Hofer's geometry (see Theorem 2. in \cite{BS25}).\\

Assume now that $(M,\omega)$ is an open and connected symplectic manifold. In this case we prove the following, slightly weaker result:

\begin{thm}\label{AuxilaryTheorem}
    Let $(M,\omega)$ be an open and connected symplectic manifold, and let $u\in C^{\infty}(M)$ be a non-constant function. There exists a constant $c=c(u)>0$ such that for any $f\in C^{\infty}_{0,c}(M)$ with $\|f\|_{\infty}\leq 1$ one can write
    \[
        f=\sum_{i=1}^N\Phi_{i,+}^*u-\Phi_{i,-}^*u,
        \qquad N\leq c(u)\cdot\bigl(\mathrm{Vol}(\mathrm{supp}\, f)+1\bigr),
    \]
    for some Hamiltonian diffeomorphisms $\Phi_{i,\pm}\in\mathrm{Ham}_c(M,\omega)$.
\end{thm}

\begin{cor}\label{FiniteVolumeCor}
    If $(M,\omega)$ has finite volume, then the constant $c(u)$ from Theorem~\ref{AuxilaryTheorem} can be replaced by a global constant $C=C(u)$ such that for any $f\in C^{\infty}_{0,c}(M)$ with $\|f\|_{\infty}\leq 1$ we have
    \[
        f=\sum_{i=1}^N\Phi_{i,+}^*u-\Phi_{i,-}^*u,
        \qquad N\leq C,
    \]
    where $\Phi_{i,\pm}\in\mathrm{Ham}_c(M,\omega)$. In particular, the number $N$ depends only on $u$ and not on $f$.
\end{cor}

\noindent Let us note that for open symplectic manifolds of infinite volume, the support of $f$ can be arbitrarily large, whereas the support of $\sum_{i=1}^N \Phi_i^* u$ has volume at most $N$ times the volume of the support of $u$. Consequently, Theorem~\hyperref[TheoremA]{A} does not hold in its original form for open symplectic manifolds of infinite volume.\\

Theorem~\ref{AuxilaryTheorem} leads to a stronger flexibility statement for the Hamiltonian adjoint action. This result will later serve as a key ingredient in our study of bi-invariant metrics on $\mathrm{Ham}(M,\omega)$ (see Sections \ref{Section:bi-invariant-metrics} and \ref{Section:exact-symp-mfd} below). We now state the main flexibility theorem:

\begin{thm}\label{MainTheorem}
    Let $(M,\omega)$ be an open and connected symplectic manifold of infinite volume, and let $u\in C^{\infty}(M)$ be a non-constant function. 
    There exists a constant $C=C(u)>0$ such that for every $f\in C^{\infty}_{0,c}(M)$ with $\|f\|_{\infty}+\|f\|_{L^1}\leq 1$,
    one can find real numbers $c_1,\ldots,c_{\ell}$ and Hamiltonian diffeomorphisms 
    $\Phi_{i,\pm}\in\mathrm{Ham}_c(M,\omega)$ satisfying
    \[
        f=\sum_{i=1}^{\ell} c_i\cdot\bigl(\Phi_{i,+}^*u-\Phi_{i,-}^*u\bigr),
        \qquad \sum_{i=1}^{\ell} |c_i|\leq C.
    \]
\end{thm}

\begin{cor}\label{Corollary:upper_bound}
    Let $\|\cdot\|$ be a $\mathrm{Ham}_c(M,\omega)$-invariant norm on $C^{\infty}_c(M)$. 
    Then there exists a constant $C>0$ such that for every $f\in C^{\infty}_c(M)$,
    \[
        \|f\|\leq C\cdot\bigl(\|f\|_{L^{\infty}} + \|f\|_{L^1}\bigr).
    \]
\end{cor}

\noindent It is not surprising that Theorem~\hyperref[TheoremA]{A} plays a central role in proving Theorems~\ref{AuxilaryTheorem} and \ref{MainTheorem}. However, in what follows we will use its local version, stated below.

\begin{thmB}[Theorem 3.1 in Buhovsky, Stokic \cite{BS25}]\label{TheoremB}
    Let $L>0$. There exists $N(n)\in\mathbb{N}$ such that for any $f\in C^{\infty}_0((-L,L)^{2n})$ with $\|f\|_{\infty}\leq L$, one can write
    \[f=\sum_{i=1}^{N(n)}\Phi_{i,+}^*\,x_1-\Phi_{i,-}^*\,x_1,\]
    for some Hamiltonian diffeomorphisms $\mathrm{Ham}_c((-8L,8L)^{2n})$.
\end{thmB}

\subsection{Bi-invariant metrics on $\mathrm{Ham}(M,\omega)$}\label{Section:bi-invariant-metrics}

Let $\|\cdot\|$ be a norm on the Lie algebra $\mathcal{A}$ of the Hamiltonian diffeomorphism group $\mathrm{Ham}(M,\omega)$. We identify $\mathcal{A}$ with the corresponding function space ($C^{\infty}_0(M)$ if $M$ is closed, and $C^{\infty}_c(M)$ if $M$ is open). Assume moreover that $\|\cdot\|$ is invariant under the adjoint action of $\mathrm{Ham}(M,\omega)$ on $\mathcal{A}$, that is,
\[
\| \mathrm{Ad}_{\varphi} f \| = \| f \circ \varphi^{-1} \| = \| f \| \quad \text{for all } \varphi \in \mathrm{Ham}(M,\omega),\, f \in \mathcal{A}.
\]

\noindent Then $\|\cdot\|$ induces a pseudo-norm $\vertiii{\,\cdot\,}$ on $\mathrm{Ham}(M,\omega)$, defined by
\[
\vertiii{\phi} = \inf \left\{ \int_0^1 \| H(t,\cdot) \| \, dt \;\middle|\; H : [0,1]\times M \to \mathbb{R},\ \phi^1_H = \phi \right\}.
\]
This pseudo-norm is conjugation-invariant, i.e.\ $\vertiii{\phi}=\vertiii{\psi\phi\psi^{-1}}$ for all $\phi,\psi\in\mathrm{Ham}(M,\omega)$. It induces a bi-invariant pseudo-metric on $\mathrm{Ham}(M,\omega)$, defined by 
\[\rho(\phi,\psi):=\vertiii{\phi\,\psi^{-1}},\quad\text{for }\phi,\psi\in\mathrm{Ham}(M,\omega).\]
Recall that the bi-invariant condition means that $\rho(\phi,\psi)=\rho(\theta\phi,\theta\psi)=\rho(\phi\theta,\psi\theta)$ for all $\phi,\psi,\theta\in\mathrm{Ham}(M,\omega)$. For $p\in[1,\infty]$, let $\rho_p$ denote the pseudo-metric on $\mathrm{Ham}(M,\omega)$ induced by the $L^p$-norm on its Lie algebra. It has been shown (see \cite{Ho90}, \cite{Vi92}) that $\rho_{\infty}$ defines a genuine metric, known as Hofer's metric, that we sometimes denote $d_{\mathrm{Hofer}}$. In \cite{EP93}, Eliashberg and Polterovich studied the pseudo-metrics $\rho_p$, and proved that for any $1 \leq p < \infty$, the pseudo-metric $\rho_p$ fails to be a genuine metric. Later, Ostrover and Wagner generalized this result as follows:

\begin{thmC}[Ostrover--Wagner \cite{OW05}] \label{ThmOW}
    Let $(M,\omega)$ be a closed symplectic manifold, and let $ \| \cdot \|$ be a $\mathrm{Ham}(M,\omega)$-invariant norm on ${\cal A}\cong C^{\infty}_0(M)$ such that $\| \cdot \| \leq C\| \cdot \|_{\infty}$ for some constant $C$, but the two norms are not equivalent. Then the associated  pseudo-metric on $\mathrm{Ham}(M,\omega)$ vanishes identically.
\end{thmC}

\noindent The theorem of Ostrover and Wagner extends to the case of open symplectic manifolds in the following way:

\begin{thm}\label{Theorem:Lower_Bound_Open_Case}
    Let $(M,\omega)$ be an open symplectic manifold, and let $\|\cdot\|$ be a $\mathrm{Ham}(M,\omega)$-invariant norm on $C^{\infty}_c(M)$. Assume there is no constant $c>0$ such that $\|F\|\geq c\,\|F\|_{\infty}$ for all $F\in C^{\infty}_{0,c}(M)$. Then the induced pseudo-metric $\rho$ on $\mathrm{Ham}(M,\omega)$ is degenerate.
\end{thm}

\noindent The proof of this result follows the same steps as in Theorem~\hyperref[thmOW]{C} and is presented in the Appendix. 
Theorem~\ref{MainTheorem}, specifically Corollary~\ref{Corollary:upper_bound}, asserts that $\|\cdot\| \le C(\|\cdot\|_{\infty} + \|\cdot\|_{L^1})$. 
The following theorem shows that, even for the largest invariant norm on $C^{\infty}_c(M)$, the induced norm on $\mathrm{Ham}(M,\omega)$ coincides with Hofer's norm when restricted to $\mathrm{ker}(\mathrm{Cal})$.

\begin{thm}\label{Theorem:Upper_Bound_Infinite_volume}
    Let $(M,\omega)$ be a connected, exact symplectic manifold. Let $\vertiii{\,\cdot\,}$ be a conjugation invariant norm on $\mathrm{Ham}(M,\omega)$ induced by a $\mathrm{Ham}(M,\omega)$-invariant norm $\|\cdot\|$, defined as
    \[\|\cdot\|:=\|\cdot\|_{\infty}+\|\cdot\|_{L^1}.\]
    Then, $\vertiii{\phi}=\|\phi\|_{\mathrm{Hofer}}$ for all $\phi\in\mathrm{ker}(\mathrm{Cal})$. 
\end{thm}

\subsection{Exact symplectic manifolds and Calabi homomorphism}\label{Section:exact-symp-mfd}

In this section we assume that the symplectic form $\omega$ is exact. Then one can define
\[\mathrm{Cal}:\mathrm{Ham}(M,\omega)\to\mathbb{R},\quad\mathrm{Cal}(\phi)=\int_{0}^{1}H(t,\cdot)\,\omega^n\,dt,\]
where $H:[0,1]\times M\to\mathbb{R}$ is a Hamiltonian function whose time--$1$ flow generates $\phi$. Since $\omega$ is exact, the value of $\mathrm{Cal}(\phi)$ is well-defined, i.e., it does not depend on the choice of Hamiltonian function $H$ with $\phi^1_H = \phi$.

Eliashberg and Polterovich showed (see Theorem~1.4.A in \cite{EP93}) that any continuous, bi-invariant, intrinsic pseudo-metric $\rho$ on an exact symplectic manifold $\mathrm{Ham}(M,\omega)$ that is not a genuine metric satisfies:
\[
\rho(\phi,\mathrm{Id}) = \mu \cdot \big\lvert \mathrm{Cal}(\phi) \big\rvert, \quad \text{for some } \mu > 0 \text{ and all } \phi \in \mathrm{Ham}(M,\omega),
\]
which classifies all degenerate pseudo-metrics. On the other hand, if one considers linear combination of $L^{\infty}$ and $L^{p}$-norms, namely $\|\cdot\|=\|\cdot\|_{\infty}+\sum_{p=1}^m\mu_p\|\cdot\|_{L^p}$, where $\mu_p\geq 0$, the induced metric $\rho$ satisfies (see Section~4.3.A. in \cite{EP93}):

\[\rho(\phi,\mathrm{Id})\geq\|\phi\|_{\mathrm{Hofer}}+\mu \cdot \big\lvert \mathrm{Cal}(\phi) \big\rvert.\]

\begin{question}[Eliashberg–Polterovich, Question~4.3.C in \cite{EP93}]
Does there exist a bi-invariant intrinsic metric on $\mathrm{Ham}(M,\omega)$ that is not equivalent (or even different) from $d_{\mathrm{Hofer}} + \mu \cdot |\mathrm{Cal}|$, where $\mu \geq 0$?
\end{question}

\noindent We answer the question by classifying, up to equivalence, all bi-invariant metrics on the group of Hamiltonian diffeomorphisms of exact symplectic manifolds:

\begin{thm}[Classification of bi-invariant pseudo-metrics on $\mathrm{Ham}(M,\omega)$]\label{Theorem:Classification_of_bi-invariant_metrics}
Let $(M,\omega)$ be a connected exact symplectic manifold, and let $\rho$ be an intrinsic bi-invariant pseudo-metric on $\mathrm{Ham}(M,\omega)$ induced by a $\mathrm{Ham}(M,\omega)$-invariant norm $\|\cdot\|$ on its Lie algebra. Then one of the following holds:
\begin{enumerate}
    \item \textbf{Degenerate case:} $\rho(\phi,\psi) = \mu\,|\mathrm{Cal}(\phi\circ\psi^{-1})|$ for some $\mu \ge 0$.
    \item \textbf{Non-degenerate case:} There exist constants $0 < c < C$ such that either
    \[
        c\,d_{\mathrm{Hofer}}(\phi,\psi) \le \rho(\phi,\psi) \le C\,d_{\mathrm{Hofer}}(\phi,\psi),
    \]
    or
    \[
        c \,\big(d_{\mathrm{Hofer}}(\phi,\psi) + |\mathrm{Cal}(\phi\circ\psi^{-1})|\big) 
        \le \rho(\phi,\psi) 
        \le C \,\big(d_{\mathrm{Hofer}}(\phi,\psi) + |\mathrm{Cal}(\phi\circ\psi^{-1})|\big).
    \]
\end{enumerate}
Thus, $\rho$ is either identically zero or, up to equivalence, coincides with one of $|\mathrm{Cal}|$, $d_{\mathrm{Hofer}}$, or $d_{\mathrm{Hofer}} + |\mathrm{Cal}|$.
\end{thm}
\section{Proof of Theorem \ref{AuxilaryTheorem}}

\begin{claim}\label{ClaimDarbouxChartForU}
    There exists a Darboux chart $(V,\varphi)$ and $\widetilde{L}>0$ with $[-\widetilde{L},\widetilde{L}]^{2n}\subset\varphi(V)$ and
    \[(u\circ\varphi^{-1})|_{[-\widetilde{L},\widetilde{L}]^{2n}}=x_1+c,\quad\text{for some constant }c\in\mathbb{R}.\]
\end{claim}
\begin{proof}
    See the proof of Lemma 4.2 in \cite{BS25}.
\end{proof}

\noindent
Define the following objects:
\begin{enumerate}
    \item Open subset $U := \varphi^{-1}([-\widetilde{L}/8, \widetilde{L}/8])\subset M$.
    \item Choose $L > 0$ sufficiently small so that $Q_L := \varphi^{-1}([-L, L]^{2n}) \subset U$.
    \item Let $h : M \to [0,1]$ be a smooth bump function with $\mathrm{supp}\,h \subset U \setminus Q_L$ and $\int_M h\,\omega^n = 1$.
    \item Let $\mathcal{C}$ be a finite set of colors with $|\mathcal{C}| = 100^n$.
\end{enumerate}

\begin{prop}\label{Proposition:Darboux_Balls_Cover}
    There exists a finite family of open Darboux balls $\mathcal{B}$ that can be split into $100^n$ disjoint families $\mathcal{B}=\bigsqcup_{c\in\mathcal{C}}\mathcal{B}_c$ such that the following is satisfied
    \begin{enumerate}
        \item $\mathrm{supp}\,f\subset U\cup\big(\bigcup_{B\in\mathcal{B}}B\big)$,
        \item for all $B\in\mathcal{B}$ we have $B\cap (Q_L\cup(\mathrm{supp}\,h))=\emptyset$, and no ball $B\in\mathcal{B}$ satisfies $B\subset U$,
        \item for every $c\in\mathcal{C}$, all the balls in $\mathcal{B}_c$ are pairwise disjoint, and no two balls $B,B'\in\mathcal{B}_c$ have a non-empty intersection with a ball from $\mathcal{B}$,
        \item for every $B\in\mathcal{B}$ there exists a sequence $\{B_i\}_{i=0}^{n_B}\subset\mathcal{B}$ such that $B_0=B$, $B_{n_B}\cap U\neq\emptyset$, and  $B_i\cap B_{i+1}\neq\emptyset$ for all $0\leq i< n_B$.
    \end{enumerate}
\end{prop}

\begin{proof}
    Choose a Riemannian metric compatible with the symplectic form on $M$. Let $\Omega\subset M$ be a bounded connected open set with $U\cup\mathrm{supp}\,f\subset\Omega$. For $\varepsilon>0$, let $\Gamma_{\varepsilon}\subset\overline{\Omega}\setminus U$ be a \emph{maximal} $\varepsilon$-separated set (i.e., distinct points are at least $\varepsilon$ apart). Define
    \[
    \mathcal{B}_{\varepsilon}:=\{B(v,\varepsilon)\mid   v\in\Gamma_{\varepsilon}\},
    \]
    where $B(v,\varepsilon)$ is the ball of radius $\varepsilon$. For $\varepsilon$ sufficiently small, these balls become Darboux balls. Let us prove that union of sets in $\mathcal{B}_{\varepsilon}$ covers $\overline{\Omega}\setminus U$. Suppose by contradiction that there exists $p\in\overline{\Omega}\setminus U$ with $d(p,\Gamma_{\varepsilon})\geq\varepsilon$. Then $p$ could be added to $\Gamma_{\varepsilon}$, contradicting maximality. Therefore, we have
    \[
    U\cup\mathrm{supp}\,f\subset\Omega\subset U\cup\bigcup_{B\in\mathcal{B}_{\varepsilon}}B.
    \]
    \noindent This construction verifies the first two properties, provided $\varepsilon>0$ is small enough. Let us prove that any $B\in\mathcal{B}_{\varepsilon}$ intersects at most $5^{2n}-1$ other balls in $\mathcal{B}_{\varepsilon}$, provided $\varepsilon>0$ is small enough. Fix $B(p,\varepsilon)\in\mathcal{B}_{\varepsilon}$ with $p\in\Gamma_{\varepsilon}$. Any ball intersecting $B(p,\varepsilon)$ has its center in $B(p,2\varepsilon)$. Consider
    \[
    \mathcal{F}:=\{B(v,\varepsilon/2)\mid v\in\Gamma_{\varepsilon}\cap B(p,2\varepsilon)\}.
    \]
    Since $\Gamma_{\varepsilon}$ is $\varepsilon$-separated, the balls in $\mathcal{F}$ are pairwise disjoint, and they all lie inside the ball $B(p,5\varepsilon/2)$. A volume comparison gives
    \[
    |\mathcal{F}|\leq\left\lfloor\frac{\mathrm{Vol}(B(p,5\varepsilon/2))}{\mathrm{Vol}(B(p,\varepsilon/2))}\right\rfloor\leq 5^{2n},
    \]
    for $\varepsilon>0$ small enough. Thus each ball in $\mathcal{B}_{\varepsilon}$ intersects at most $5^{2n}-1$ others. Consider the graph $\mathcal{G}$ whose vertices are the sets in $\mathcal{B}_{\varepsilon}$, where two vertices are connected by an edge if and only if the corresponding sets have non-empty intersection. It follows that the degree of every vertex in $\mathcal{G}$ is at most $d = 5^{2n} - 1$. Hence, the vertices of $\mathcal{G}$ can be colored with at most $d^2 + 1 < 100^n$ colors, so that no two vertices of the same color are at distance $1$ or $2$ in $\mathcal{G}$. This establishes the second property. Finally, since $\Omega \supset U \cup \mathrm{supp}\,f$ is connected, it follows that $\mathcal{G}$ is connected, which proves the third property.
\end{proof}

\begin{lemma}\label{Lemma:transition_charts_lemma}
Let $\{(U_i,\varphi_i)\}_{i=1}^{m}$ be a finite family of Darboux charts on $M$. Then we can modify each chart $\varphi_i$ to $\varphi'_i$ (by modifying it only on intersections with other charts) so that the family $\{(U_i,\varphi'_i)\}_{i=1}^{m}$ remains a family of Darboux charts and satisfies the following: whenever $U_i \cap U_j \neq \emptyset$, there exists an open subset $B_{ij} \subset U_i \cap U_j$ on which the transition map is the identity, i.e.,
\[
\varphi'_i \circ (\varphi'_j)^{-1}\big|_{B_{ij}} = \mathrm{Id}.
\]
\end{lemma}

\begin{proof}
    For every ordered pair $(i,j)$ with $U_i\cap U_j\neq 0$ pick a point $p_{ij}\in U_i\cap U_j$, such that $p_{ij}\neq p_{kl}$ whenever $(i,j)\neq(k,l)$. Denote $\varphi_{ij}:=\varphi_i\circ\varphi_j^{-1}$. Without loss of generality, we may assume that $\varphi_{ij}(p_{ij})=p_{ij}$.
    \begin{claim}\label{Claim:Moser's_trick}
    For every ordered pair $(i,j)$ with $U_i \cap U_j \neq \emptyset$, and for every open subset $V_{ij} \subset U_i \cap U_j$ containing a point $p_{ij}$, there exists a symplectomorphism 
    \[
        \psi_{ij} \in \mathrm{Symp}_c(V_{ij})
    \] 
    such that $\psi_{ij}$ coincides with $\varphi_{ij}$ on an open neighbourhood of $p_{ij}$.
    \end{claim}

\begin{proof}[Proof of Claim \ref{Claim:Moser's_trick}]
The linear symplectic map $d\varphi_{ij}(p_{ij})$ can be realized as the time--1 flow of a quadratic Hamiltonian vector field, and hence it fixes $p_{ij}$.  
By Moser’s method one then obtains, in a neighborhood of $p_{ij}$, a Hamiltonian isotopy $(\varphi_t)_{t\in[0,1]}$ with $\varphi_0=\mathrm{id}$, $\varphi_1=\varphi_{ij}$, and each $\varphi_t$ fixing $p_{ij}$.
Thus $\varphi_{ij}$ agrees near $p_{ij}$ with the time--1 map of a Hamiltonian flow generated by some $H_t$.  
Finally, multiplying $H_t$ by a cut--off function supported in $V_{ij}$ and equal to $1$ near $p_{ij}$ produces compactly supported Hamiltonians whose time--1 flow $\psi_{ij}$ still fixes $p_{ij}$ and coincides with $\varphi_{ij}$ on a neighborhood of $p_{ij}$.
\end{proof}

\noindent For every ordered pair $(i,j)$, with $U_i \cap U_j \neq \emptyset$, pick a small open subset $V_{ij} \subset U_i \cap U_j$, containing a point $p_{ij}$, such that $V_{ij}\cap V_{kl}=\emptyset$ whenever $(i,j)\neq (k,l)$. Apply the above Claim to get symplectomorphisms $\psi_{ij}$. Finally, define
\[
\varphi'_j(p) =
\begin{cases}
\psi_{ij} \circ \varphi_j(p), & \text{if } U_i \cap U_j \neq \emptyset \text{ and } p \in V_{ij}, \\
\varphi_j(p), & \text{otherwise.}
\end{cases}
\]
\noindent We now verify that on $V_{ij}$ we have
\[
\varphi'_i \circ (\varphi'_j)^{-1} = \varphi_i \circ (\psi_{ij} \circ \varphi_j)^{-1} = \varphi_{ij} \circ \psi_{ij}^{-1}.
\]  
Since $\psi_{ij}$ coincides with $\varphi_{ij}$ on a small open neighbourhood of $p_{ij} \in V_{ij}$, we complete the proof by taking $B_{ij}$ to be an open neighbourhood of $p_{ij}$ where this equality holds.
\end{proof}

\noindent Before we proceed with the proof, we apply the Lemma \ref{Lemma:transition_charts_lemma} to the family of Darboux charts consisting of $U$ and balls $B\in\mathcal{B}$. Now we use a partition of unity to decompose
\[
f = f_U + \sum_{B \in \mathcal{B}} f_B,
\]
where $\mathrm{supp}\,f_U \subset U$, $\|f_U\|_{\infty} \leq 1$, 
and for each $B \in \mathcal{B}$ we have $\mathrm{supp}\,f_B \subset B$ and $\|f_B\|_{\infty} \leq 1$.

\begin{claim}\label{Claim:Cover_by_cubes_in_charts}
    For every pair $(c,\lambda)\in\mathcal{C}\times\mathcal{X}$, where $\mathcal{X}=\{0,1\}^{2n}$, there exists $a>0$, and a finite collection of disjoint open sets $\mathcal{Q}_{c}^{\lambda}$, such that the following holds:
    \begin{enumerate}
        \item $\bigcup_{B\in\mathcal{B}}\mathrm{supp}\,f_{B}\subset \bigcup_{(c,\lambda)\in\mathcal{C}\times\mathcal{X}}\bigsqcup_{Q\in\mathcal{Q}_{c}^{\lambda}}Q$,
        \item $\mathrm{Vol}\big(\bigsqcup_{Q\in\mathcal{Q}_{c}^{\lambda}}Q\big)\leq 2\mathrm{Vol}(\mathrm{supp}\,f)$ for all $(c,\lambda)\in\mathcal{C}\times\mathcal{X}$,
        \item for every $Q\in\mathcal{Q}_{c}^{\lambda}$, we have $Q\subset B\in\mathcal{B}_c$ and inside the Darboux chart $B$ the image of $Q$ has form $v+(-2a/3,2a/3)^{2n}$ for some vector $v\in\mathbb{R}^{2n}$.
    \end{enumerate}
\end{claim}

\begin{proof}
    Fix a Darboux ball $B\in\mathcal{B}_c$ with the chart map $\varphi_{B}:B\to\varphi_B(B)\subset\mathbb{R}^{2n}$. Denote $\Omega:=\varphi_B(\mathrm{supp}\,f_B)$ and let $V\subset\varphi_B(B)$ be an open neighbourhood of $\Omega$. Pick a $\delta>0$ small enough so that 
    \[V_{\delta}=\{x\in\mathbb{R}^{2n}\mid d(x,\Omega)\leq\delta\}\subset V.\]
    
    \noindent Let $0<a<\frac{\delta}{2n}$, and for each $\lambda\in\mathcal{X}=\{0,1\}^{2n}$ we define a finite grid $\Gamma^{a}_{\lambda}\subset V_{\delta}$ as
    \[\Gamma^{a}_{\lambda}:=(a\cdot\lambda+2a\mathbb{Z}^{2n})\cap V_{\delta}.\]

    \noindent For each $\lambda\in\mathcal{X}$ we define collection of open cubes
    \[\mathcal{Q}^{B}_{\lambda}:=\{v+(-2a/3,2a/3)^{2n}\mid v\in\Gamma^{a}_{\lambda}\text{ and }\{v+(-2a/3,2a/3)^{2n}\}\cap\Omega\neq\emptyset\}.\]
    
    \noindent Moreover, define $\mathcal{Q}^{B}$ to be the union of all cubes from different collections. We claim that
    
    \begin{equation}\label{Covering_Equation}
    \Omega\subset\bigcup_{Q\in\mathcal{Q}^{B}}Q=\bigcup_{\lambda\in\mathcal{X}}\bigcup_{v\in\Gamma^{a}_{\lambda}}v+(-2a/3,2a/3)^{2n}.
    \end{equation}
    
    \noindent Take a point $p=(p_1,p_2,\ldots,p_{2n})\in \Omega$ and define $q=(a\cdot\lfloor p_1/a\rfloor,\ldots,a\cdot\lfloor p_{2n}/a\rfloor)\in a\cdot\mathbb{Z}^{2n}$. For each $\lambda\in\mathcal{X}=\{0,1\}^{2n}$ define point $q_{\lambda}=q+a\cdot\lambda$. Points $\{q_{\lambda}\mid\lambda\in\mathcal{X}\}$ are corners of a cube of side $a$ and the point $p$ is inside this cube, which implies that $p$ is covered by
    \[\bigcup_{\lambda\in\mathcal{X}}q_{\lambda}+(-2a/3,2a/3)^{2n}.\]
    \noindent To show (\ref{Covering_Equation}) it only remains to prove $\{q_{\lambda}\mid\lambda\in\mathcal{X}\}\subset\Gamma:=\bigcup_{\lambda\in\mathcal{X}}\Gamma^{a}_{\lambda}$, which is equivalent to showing that $\{q_{\lambda}\mid\lambda\in\mathcal{X}\}\subset V_{\delta}$. If $q_{\lambda}\in\Omega\subset V_{\delta}$ we are done, otherwise if $q_{\lambda}\not\in\Omega$ we have $d(q_{\lambda},\partial\Omega)<2n\cdot a<\delta$ which implies $q_{\lambda}\in V_{\delta}$ and proves (\ref{Covering_Equation}). Moreover, note that every cube in $\mathcal{Q}^a$ has its center in $V_{\delta}$ and diameter less than $2n\cdot a<\delta$, hence 
    \[
    \bigcup_{Q\in\mathcal{Q}^a}Q\subset\{x\mid d(x,\Omega)<2\delta\}\subset V\subset\varphi_{B}(B)
    \]
    for sufficiently small $\delta$. Finally, define
    \[\mathcal{Q}_{c}^{\lambda}:=\bigcup_{B\in\mathcal{B}_{c}}\bigcup_{Q\in\mathcal{Q}^{B}_{\lambda}}\varphi_B^{-1}(Q).\]
    One checks that families $\mathcal{Q}^{\lambda}_{c}$ satisfy desired properties.
\end{proof}

\noindent Once again, using the partition of unity, we can write
\[
f = f_U + \sum_{B \in \mathcal{B}} f_B = f_U + \sum_{(c,\lambda) \in \mathcal{C} \times \mathcal{X}} f_{c,\lambda},
\]
where each $f_{c,\lambda} \in C_c^{\infty}(M)$ satisfies 
$\mathrm{supp}\,f_{c,\lambda} \subset \bigcup_{Q \in \mathcal{Q}_c^{\lambda}} Q$, $\|f_{c,\lambda}\|_{\infty} \leq 1$ for all $(c,\lambda) \in \mathcal{C} \times \mathcal{X}$.

\begin{claim}
    Let $N_L:=3\mathrm{Vol}(\mathrm{supp}\,f)/(2L)^{2n}$. For each $(c,\lambda)\in\mathcal{C}\times\mathcal{X}$, the family of cubes $\mathcal{Q}_{c}^{\lambda}$ can be split into $N_L$ disjoint families $\big\{\mathcal{F}^{c,\lambda}_i\big\}_{i=1}^{N_L}$  such that for each $1\leq i\leq N_L$ there exists a Hamiltonian diffeomorphism $\Psi_i\in\mathrm{Ham}_c((M\setminus\mathrm{supp}\,h),\,\omega)$ with
    \[\bigsqcup_{Q\in\mathcal{F}^{c,\lambda}_i}\Psi_i(Q)\subset Q_L.\]
\end{claim}

\begin{proof}
    Consider the graph $\mathcal{G}$ whose vertices are the Darboux balls $B \in \mathcal{B}$ together with the set $U$, where two vertices are connected by an edge if their corresponding sets have non-empty intersection. Proposition~\ref{Proposition:Darboux_Balls_Cover} implies that $\mathcal{G}$ is connected. Fix the vertex corresponding to $U$, and for any $B \in \mathcal{B}_c$ consider the shortest path $B_0 = B, B_1, \dots, B_m = U$ from $B$ to $U$. Lemma~\ref{Lemma:transition_charts_lemma} guarantees that for each $0 \le i < m$ there exists an open subset $B_{i(i+1)} \subset B_i \cap B_{i+1}$ on which the transition map from $B_i$ to $B_{i+1}$ restricts to the identity. In particular, any sufficiently small standard cube in $B_i$ can be mapped, via a Hamiltonian diffeomorphism, to a standard cube in $B_{i+1}$ by composing a translation with the chart transition map. Consequently, any sufficiently small standard cube in $B$ can be transported to $B_m = U$ so that its image is a standard cube in the chart $U$. Moreover, all cubes $Q \in \mathcal{Q}^{\lambda}_{c}$ with $Q \subset B$ can be arranged in a sequence so that they can be transported one by one to $U$ via Hamiltonian isotopies that fix all other cubes in $B$ while transporting a given cube as described above. Therefore, for every $(c,\lambda) \in \mathcal{C} \times \mathcal{X}$ and $B \in \mathcal{B}_c$, this procedure defines a total order on the set 
    \[
    \{Q \in \mathcal{Q}^{\lambda}_{c} \mid Q \subset B\}.
    \]
    Let $\mathcal{T}$ be a shortest-path spanning tree of $\mathcal{G}$ rooted at the vertex corresponding to the set $U$. Fix $(c,\lambda) \in \mathcal{C} \times \mathcal{X}$, and consider all cubes in $\mathcal{Q}^{\lambda}_{c}$. Let $B_1, B_2, \dots, B_k \in \mathcal{B}$ be the leaves of $\mathcal{T}$. For any $Q, Q' \in \mathcal{Q}^{\lambda}_{c}$, let $Q \subset B_Q \in \mathcal{B}_c$ and $Q' \subset B_{Q'} \in \mathcal{B}_c$. To each vertex corresponding to $B_Q, B_{Q'}$ we assign unique numbers $1 \le k_Q, k_{Q'} \le k$ such that the path from the leaf $B_{k_Q}$ to the root $U$ contains $B_Q$ (and similarly for $k_{Q'}$). We then define a total order on $\mathcal{Q}^{\lambda}_{c}$ by defining $Q \prec Q'$ if:
    \begin{enumerate}
        \item $k_Q < k_{Q'}$, or
        \item $k_Q = k_{Q'}$ and the distance from $B_Q$ to the root $U$ is less than that of $B_{Q'}$, or
        \item $k_Q = k_{Q'}$, $B_Q = B_{Q'}$, and $Q$ precedes $Q'$ in the order previously defined on $B_Q$.
    \end{enumerate}
    Finally, we partition the family of disjoint cubes $\mathcal{Q}^{\lambda}_c$, respecting the order $\prec$, into $N_L$ subfamilies $\{\mathcal{F}^{c,\lambda}_i\}_{i=1}^{N_L}$ so that the total volume of cubes in each subfamily does not exceed $\frac{1}{2}\mathrm{Vol}(Q_L) = \frac{1}{2} (2L)^{2n}$. Fix a family of cubes $\mathcal{F}^{c,\lambda}_i = \{Q_1, \ldots, Q_l\}$ whose indices respect the established order. We have already explained how a single cube can be transported to the set $U$, and from there further translated to $Q_L$ while avoiding $\mathrm{supp}\,h$. 

    Assume that the cubes $Q_1, \ldots, Q_i$ have already been transported to $Q_L$ via a Hamiltonian isotopy that fixes $Q_{i+1}, \ldots, Q_l$ and $\mathrm{supp}\,h$. We now show that $Q_{i+1}$ can also be transported to $Q_L$ via a Hamiltonian isotopy fixing $Q_{i+2}, \ldots, Q_l$ and $\mathrm{supp}\,h$. 

    The cube $Q_{i+1}$ belongs to a unique ball $B \in \mathcal{B}_c$. Let $B_0 = B, B_1, \ldots, B_m = U \subset \mathcal{B}$ be the path in the minimal spanning tree $\mathcal{T}$ connecting $B$ to $U$. By the construction of the order $\prec$ and by the third property of Proposition~\ref{Proposition:Darboux_Balls_Cover}, we have
    \[
    (B_1 \cup \cdots \cup B_m) \cap (Q_{i+2} \cup \cdots \cup Q_l) = \emptyset.
    \]
    Hence, we can transport the cube $Q_{i+1}$ to $Q_L$ as described above, without the risk of intersecting any of the remaining cubes along the way. The volume assumption ensures that all cubes fit inside the large cube $Q_L$.
\end{proof}

\noindent For $F\in C^{\infty}_c(M)$ denote $S(F):=\int_MF\,\omega^n$. Since $S(f)=0$ we can write:
\[f=f_U+\sum_{(c,\lambda)\in\mathcal{C}\times\mathcal{X}}\sum_{i=1}^NF^{c,\lambda}_i=(f_U-S(f_{U}) h)+\sum_{(c,\lambda)\in\mathcal{C}\times\mathcal{X}}\sum_{i=1}^{N_L}(F^{c,\lambda}_i-S(F^{c,\lambda}_i)h),\]
where $F^{c,\lambda}_i(x)=f_{c,\lambda}(x)$ for  $x\in\bigsqcup_{Q\in\mathcal{F}^{c,\lambda}_i}Q$ and $0$ otherwise. Note that
\begin{itemize}
    \item $(f_U-S(f_U)h)\in C^{\infty}_{0,c}(U)$ and $\|f_U-S(f_U)h\|_{\infty}\leq 2$, 
    \item $\Psi_i^*(F^{c,\lambda}_i-S(F^{c,\lambda}_i)h)\in C^{\infty}_{0,c}(U)$ and $\|\Psi_i^*(F^{c,\lambda}_i-S(F^{c,\lambda}_i)h)\|_{\infty}\leq 1$.
\end{itemize}

\noindent Finally, we apply Theorem~\hyperref[TheoremB]{B} to the function $(f_U - S(f_U)h)/\lceil\frac{2}{L}\rceil$, as well as to each of the functions $\bigl(\Psi_i^*(F^{c,\lambda}_i - S(F^{c,\lambda}_i)h)\bigr)/\lceil\frac{1}{L}\rceil$ for $1 \leq i \leq N_L$ and $(c,\lambda) \in \mathcal{C} \times \mathcal{X}$, which gives us desired representation for 
\[
\begin{aligned}
N &= \Big\lceil \frac{2}{L} \Big\rceil \cdot N(n)
   + \Big\lceil \frac{1}{L} \Big\rceil \cdot |\mathcal{X} \times \mathcal{C}| \cdot N_L \cdot N(n) \\[4pt]
  &= \Big\lceil \frac{2}{L} \Big\rceil \cdot N(n)
   + \Big\lceil \frac{1}{L} \Big\rceil \cdot 100^{2n} \cdot N(n) \cdot \tfrac{3}{L^{2n}} \cdot \mathrm{Vol}(\mathrm{supp}\,f),
\end{aligned}
\]
where $N(n)$ is the number from the Theorem~\hyperref[TheoremB]{B}, and $L$ depends only on $u$.\qed

\section{Proof of Theorem \ref{MainTheorem}}

The proof of Theorem \ref{MainTheorem} is consequence of Theorem \ref{AuxilaryTheorem} and the following proposition.

\begin{prop}\label{PropositionL1-Splitting}
    Let $(M,\omega)$ be an open symplectic manifold, and let $f\in C^{\infty}_{0,c}(M)$ with $\|f\|_{\infty}+\int_{M}|f|\omega^n\leq 1$. There exists a finite sequence of functions $f_1,f_2,\ldots,f_m\in C^{\infty}_{0,c}(M)$ such that $f=\sum_{i=1}^{m} f_i$ and $\sum_{i=1}^{m}\|f_i\|_{\infty}\cdot(\mathrm{Vol}(\mathrm{supp}\, f_i)+1)\leq 100$.
\end{prop}

\noindent First apply Proposition \ref{PropositionL1-Splitting} to get functions $f_1,\ldots,f_m\in C^{\infty}_{0,c}(M)$ with $f=\sum_{i=1}^mf_i$ and $\sum_{i=1}^m\|f_i\|_{\infty}\cdot(\mathrm{Vol}(\mathrm{supp}\,f)+1)\leq 100$. Next, we apply Theorem \ref{AuxilaryTheorem} to each function $f_i/\|f_i\|_{\infty}$ for $1\leq i\leq m$. We get
    \[(\forall 1\leq i\leq m)\quad\frac{f_i}{\|f_i\|_{\infty}}=\sum_{j=1}^{N_i}\Phi_{j,+}^*u-\Phi_{j,-}^*u,\text{ with }N_i\leq c(u)\cdot (\mathrm{Vol}(\mathrm{supp}\,f_i)+1).\]
    \noindent Finally, we can write
    \[f=\sum_{i=1}^m\sum_{j=1}^{N_i}\|f_i\|_{\infty}\cdot(\Phi_{j,+}^*u-\Phi_{j,-}^*u),\]
    where $\sum_{i=1}^m N_i\|f_i\|_{\infty}\leq c(u)\cdot\sum_{i=1}^n\|f_i\|_{\infty}\cdot(\mathrm{Vol}(\mathrm{supp}\,f_i)+1)\leq 100\cdot c(u)$.\qed
    
\subsection{Proof of Proposition \ref{PropositionL1-Splitting}}

    Assume that the function $f$ satisfies the following condition:
    
    \begin{equation}\label{condition}
        a\neq 0\implies \mathrm{Vol}(f^{-1}(\{a\}))=0
    \end{equation}
    
    \noindent We inductively construct a decreasing sequence $a_0>a_1>a_2>\ldots>0$ of positive numbers in the following way: we set $a_0=1$, and once we have constructed $a_0,a_1,\ldots,a_i$ we define $a_{i+1}$ as follows
    \begin{itemize}
        \item If $\mathrm{Vol}\,\{x\in M\mid a_i/2< |f(x)|\leq a_i\}<1$ we define $a_{i+1}:=a_i/2$,
        \item otherwise, we define $a_{i+1}$ by the equation $\mathrm{Vol}\,\{x\in M\mid a_{i+1}< |f(x)|\leq a_{i}\}=1$.
    \end{itemize}

    \noindent Assumption at the beginning ensures that the function 
    \[t\mapsto\mathrm{Vol}\,\{x\in M\mid t<|f(x)|\leq a\}\]
    is continuous, and hence the sequence is well defined. Define sets
    \[S_i:=\{x\in M\mid a_{i}<|f(x)|\leq a_{i-1}\}.\]
    \noindent Note that $\mathrm{supp}\,f=\bigsqcup_{i=1}^{\infty} S_i$,  $a_{i+1}\geq\frac{a_i}{2}$ and $\mathrm{Vol}(S_i)\leq 1$. Additionally we have

    \begin{equation}\label{ineq1}
        1\geq\int_{M}|f|\omega^n=\sum_{i=1}^{\infty}\int_{S_i}|f|\omega^n\geq\sum_{i=1}^{\infty}a_{i}\cdot\mathrm{Vol}(S_i)\geq\sum_{i=1}^{\infty}\frac{a_{i-1}}{2}\cdot\mathrm{Vol}(S_i)
    \end{equation}

    \noindent Let $k_1 < k_2 < \ldots < k_m$ be the indices for which $\mathrm{Vol}(S_{k_i}) = 1$. Since $f$ has compact support, only finitely many such indices exist. For all other indices $i$, we have the relation $a_{i+1} = a_i/2$. Hence we obtain
    \begin{equation}\label{ineq2}
        \begin{split}
            \sum_{i=0}^{\infty}a_i & \leq \sum_{j=0}^{\infty}\frac{a_0}{2^j}+\sum_{j=1}^m\sum_{i=0}^{\infty}\frac{a_{k_j}}{2^i}=2+2\,\sum_{j=1}^{m}a_{k_j}= 2+2\,\sum_{j=1}^{m}a_{k_j}\cdot\mathrm{Vol}(S_{k_j}) \\
            & \leq 2+2\int_{S_{k_1}\cup S_{k_2}\cup\ldots\cup S_{k_m}}|f|\,\omega^n\leq 2+2\int_{M}|f|\,\omega^n\leq 4.
        \end{split}
    \end{equation}
    
    \noindent Let $h \in C^{\infty}_c(M \setminus \mathrm{supp}\,f)$ be a function with $\int_{M} h\,\omega^n = 1$, $\|h\|_{\infty} = 1$, and $\mathrm{Vol}(\mathrm{supp}\,h) = 2$ (it exists since $\mathrm{Vol}(M)=\infty$). Define a sequence of functions
    \[(\forall\, i\in\mathbb{N})\quad\widetilde{f}_i:=f|_{S_i}-h\cdot\int_{S_i}f\,\omega^n,\]

    \noindent where $f|_{S_i}(x):=f(x)$ for $x\in S_i$ and $f|_{S_i}(x):=0$ for $x\in M\setminus S_i$. Note that the functions $\widetilde{f}_i$ are discontinuous, but before addressing this let us establish some properties. First,  
    \[
    f=\sum_{i=1}^{\infty}f|_{S_i}
    =\sum_{i=1}^{\infty}f|_{S_i}-h\cdot\int_{M}f\,\omega^n
    =\sum_{i=1}^{\infty}\Big(f|_{S_i}-h\cdot\int_{S_i}f\,\omega^n\Big)
    =\sum_{i=1}^{\infty}\widetilde{f}_i,
    \]
    \[
    (\forall\,i\in\mathbb{N})\quad
    \int_{M}\widetilde{f}_i\,\omega^n
    =\int_{S_i}f\,\omega^n
    -\int_{M}h\,\omega^n\cdot\int_{S_i}f\,\omega^n
    =0.
    \]

    \noindent Moreover,
    \begin{equation}\label{ineq3}
    \begin{aligned}
        &\|\widetilde{f}_i\|_{\infty}
        =\max\Big\{a_{i-1},\Big|\int_{S_i}f\,\omega^n\Big|\Big\}
        \leq\max\{a_{i-1},a_{i-1}\cdot\mathrm{Vol}(S_i)\}
        =a_{i-1}, \\[1ex]
        &\mathrm{Vol}(\mathrm{supp}\,\widetilde{f}_i)
        =\mathrm{Vol}(S_i)+\mathrm{Vol}(\mathrm{supp}\,h)
        =\mathrm{Vol}(S_i)+2.
    \end{aligned}
    \end{equation}

    \noindent Using inequalities (\ref{ineq1}),(\ref{ineq2}) and (\ref{ineq3}) we obtain
    \begin{equation}\label{ineq4}
        \begin{split}
            \sum_{i=1}^{\infty}\|\widetilde{f}_i\|_{\infty}\cdot(\mathrm{Vol}(\mathrm{supp}\,\widetilde{f}_i)+1) & \leq\sum_{i=1}^{\infty}a_{i-1}(\mathrm{Vol}(S_i)+\mathrm{Vol}(\mathrm{supp}\,h)+1) \\
            & =\sum_{i=1}^{\infty}a_{i-1}\cdot\mathrm{Vol}(S_i)+3\cdot\sum_{i=0}^{\infty}a_i\leq 14
        \end{split}
    \end{equation}

    \noindent Next, note that $\sum_{i=1}^{\infty}\mathrm{Vol}(S_i)=\mathrm{Vol}(\mathrm{supp}\,f)<\infty$, so there exists an index $m\in\mathbb{N}$ such that $\sum_{i=m}^{\infty}\mathrm{Vol}(S_i)<1$. Define the set $S_{\infty}:=\bigcup_{i=m}^{\infty}S_i$ and the function
    \[
    \widetilde{f}_{\infty} := \sum_{i=m}^{\infty} \widetilde{f}_i = f|_{S_{\infty}} - h \cdot \int_{S_{\infty}} f\,\omega^n.
    \]

    \noindent Note that $\|\widetilde{f}_{\infty}\|_{\infty}\leq 1$, $\mathrm{Vol}(S_{\infty})=\sum_{i=m}^{\infty}\mathrm{Vol}(S_i)<1$ and $\mathrm{supp}\,\widetilde{f}_{\infty}=S_{\infty}\sqcup\mathrm{supp}\,h$, therefore we have
    \[f=\widetilde{f}_{\infty}+\sum_{i=1}^{m-1}\widetilde{f}_i,\]
    \[\|\widetilde{f}_{\infty}\|_{\infty}(\mathrm{Vol}(\mathrm{supp}\,f_{\infty})+1)+\sum_{i=1}^{m-1}\|\widetilde{f}_i\|_{\infty}(\mathrm{Vol}(\mathrm{supp}\,\widetilde{f}_i)+1)\leq 4+14=18.\]
    It only remains to modify the functions $\widetilde{f}_1, \widetilde{f}_2, \ldots, \widetilde{f}_{m-1}, \widetilde{f}_{\infty}$ to obtain smooth functions $f_1, f_2, \ldots, f_m$ with the same sum, such that the $C^0$-norms and the volumes of the supports change by a sufficiently small amount. Fix a Riemannian distance $d$ on $M$ and let $\varepsilon > 0$. For each $1 \leq i \leq m-1$, let $\chi_i: M \to [0,1]$ be smooth with $\chi_i|_{S_i} = 1$ and $\chi_i(x) = 0$ whenever $d(x, \overline{S_i}) \geq \varepsilon$. Finally, define
    \[
    (\forall\, 1 \leq i \leq m-1)\quad f_i := \chi_i f - h \cdot \int_{M} \chi_i f \, \omega^n,
    \]
    \[
        f_m := f - \sum_{i=1}^{m-1} f_i 
       = f \cdot \Big(1 - \sum_{i=1}^{m-1} \chi_i\Big)  + h \cdot \int_{M} f \cdot \Big(\sum_{i=1}^{m-1} \chi_i\Big) \, \omega^n.
    \]

    \noindent For every $i \in \{1,2,\ldots,m-1\}$ we have $\|f_i\|_{\infty} \to \|\widetilde{f}_i\|_{\infty}$ and $\mathrm{Vol}(\mathrm{supp}\,f_i) \to \mathrm{Vol}(\mathrm{supp}\,\widetilde{f}_i)$ as $\varepsilon \to 0$; therefore, for $\varepsilon > 0$ small enough we have
    \begin{equation}
    \label{ineq:sum_1_to_(m-1)}
        \sum_{i=1}^{m-1}\|f_i\|_{\infty}\cdot(\mathrm{Vol}(\mathrm{supp}\,f_i)+1)\leq 1+\sum_{i=1}^{m-1}\|f_i\|_{\infty}\cdot(\mathrm{Vol}(\mathrm{supp}\,\widetilde{f}_i)+1)\leq 19.
    \end{equation}

    \noindent Next, note that $\mathrm{supp}\,f\cdot\big(1-\sum_{i=1}^{m-1}\chi_i\big)\subset(\mathrm{supp}\,f)\setminus\bigcup_{i=1}^{m-1}S_i=S_{\infty}$, hence we get
    \[\mathrm{Vol}(\mathrm{supp}\,f_m)\leq\mathrm{Vol}( S_{\infty}\sqcup\mathrm{supp}\,h)\leq 1+2=3.\]
    
    \noindent Additionally, since $\|f_{m}\|_{\infty}\leq 1$, we obtain

    \begin{equation}\label{ineq:f_m*Vol(supp)}
        \|f_m\|_{\infty}\cdot(\mathrm{Vol}(\mathrm{supp}\,f_m)+1)\leq 4.
    \end{equation}
    Combining (\ref{ineq:sum_1_to_(m-1)}) and (\ref{ineq:f_m*Vol(supp)}) we get
    \[\sum_{i=1}^m\|f_i\|_{\infty}\cdot(\mathrm{Vol}(\mathrm{supp}\,f_i)+1)\leq 19+4=23,\]

    \noindent which finishes the proof in the case when $f$ satisfies (\ref{condition}). If $f$ does not satisfy (\ref{condition}), then there exists an arbitrarily $C^0$-small function $g$, supported in a slightly larger neighbourhood of $\mathrm{supp}\,f$, such that $f-g$ satisfies (\ref{condition}). Hence we can write $f-g=\sum_{i=1}^n f_i$ as above. Finally, by choosing $g$ so that $\|g\|_{\infty}\cdot \big(\mathrm{Vol}(\mathrm{supp}\,g)+1\big)<1$, the result follows.\qed

\section{Proof of Theorem \ref{Theorem:Upper_Bound_Infinite_volume}}

\begin{proof}[Proof of Theorem \ref{Theorem:Upper_Bound_Infinite_volume}]
    Let $H$ be a Hamiltonian function with $\phi^1_H = \phi$. By Lemma~\ref{Lemma:Reduction} and Proposition~\ref{Proposition:first_reduction}, we obtain
    \[
    \vertiii{\phi^1_H} = \vertiii{\phi^1_G} \leq \|G\|_{L^{(1,\infty)}} = \|H\|_{L^{(1,\infty)}}.
    \]
    Taking the infimum over all Hamiltonians $H$ generating $\phi$ completes the proof.
\end{proof}

\begin{lemma}\label{Lemma:Reduction}
Let $(M,\omega)$ be a symplectic manifold of infinite volume, and let $\phi^1_H$ be the time--$1$ map of a compactly supported Hamiltonian $H:[0,1]\times M \to \mathbb{R}$ normalized by
\[
\int_0^1 \int_M H \,\omega^n\,dt = 0.
\]
Then, for any $\varepsilon>0$, there exists a compactly supported Hamiltonian $G:[0,1]\times M \to \mathbb{R}$ such that
\begin{enumerate}[label=(\roman*)]
    \item $\phi^1_G = \phi^1_H$,
    \item $\int_M G(t,\cdot)\,\omega^n = 0$ for all $t \in [0,1]$,
    \item $\|G\|_{L^{(1,\infty)}} = \|H\|_{L^{(1,\infty)}}$.
\end{enumerate}
\end{lemma}

\begin{proof}

\noindent For $t \in [0,1]$, set
\[
a(t) := \int_{M} H(t,\cdot)\,\omega^n, \quad \text{so that } \int_0^1 a(t)\,dt = 0.
\]
Let $V \subset M$ be a finite-volume subset with $\mathrm{supp}\,H(t,\cdot) \subset V$ for all $t$. Choose a bump function $\chi:M\to [0,1]$ such that
\[
\mathrm{supp}(\chi) \cap \mathrm{supp}(H) = \emptyset, \quad 
\|\chi\|_{\infty} < \frac{1}{\mathrm{Vol}(V)}, \quad 
\int_{M} \chi\,\omega^n = 1,
\]
which exists since $H$ is compactly supported and $M$ has infinite volume. Let $\phi_K^t$ be the Hamiltonian flow of
\[
K(t,x) := -a(t)\,\chi(x),
\]
so that $\phi^1_K = \mathrm{Id}$. The triangle inequality gives $\|H(t,\cdot)\|_{\infty} \geq |K(t,x)|$. Define
\[
\Phi^t := \phi_K^t \circ \phi_H^t, \qquad G(t,x) := H(t,x) - a(t)\,\chi(x),
\]
so that $\Phi^1 = \phi^1_G = \phi^1_H$ and $\int_M G(t,\cdot)\,\omega^n = 0$ for all $t$.\\

\noindent Since $\|H(t,\cdot)\|_{\infty} \geq \|K(t,\cdot)\|_{\infty}$ and the supports of $H$ and $K$ are disjoint, we have
\[
\|G\|_{L^{(1,\infty)}} = \int_0^1 \|G(t,\cdot)\|_{\infty}\,dt = \int_0^1 \|H(t,\cdot)\|_{\infty}\,dt = \|H\|_{L^(1,\infty)},
\]
which completes the proof.
\end{proof}

\begin{prop}\label{Proposition:first_reduction}
    Let $G:[0,1]\times M \to \mathbb{R}$ be a compactly supported Hamiltonian function such that
    \[
        \int_{M} G(t,\cdot)\,\omega^n = 0 \quad \text{for all } t \in [0,1].
    \]
    Then $\vertiii{\phi^1_G} \leq \|G\|_{L^{(1,\infty)}}$.
\end{prop}

\begin{proof}
Split the interval $[0,1]$ into $N$ intervals
\[I_i:=\Big[\frac{i-1}{N},\frac{i}{N}\Big],\quad i\in\{1,\ldots,N\}\]

\noindent For each $i\in\{1,\ldots,N\}$ define a function $g_i\in C^{\infty}_c(M)$ as the average of $G$ over $I_i$:
\[g_i(x):=\int_{I_i} G(t,x)\,dt.\]

\noindent Let $\chi_i:[0,1]\to[0,\infty)$ be a smooth bump function with the support in $I_i$ which satisfies
\begin{equation}\label{Equation:Bump_function_condition}
    \int_{I_i}\chi_i(t)\,dt=1,\quad\text{and}\quad\int_{I_i}\Big\lvert 1-\frac{\chi_i(t)}{N}\Big\rvert\,dt<\frac{1}{N^2}.
\end{equation}

\noindent Let $\widetilde{G}:[0,1]\times M\to\mathbb{R}$ be a smooth time-dependent Hamiltonian function defined as
\[\widetilde{G}(t,x):=\sum_{i=1}^N \chi_i(t)\,g_i(x).\]
Since $\int_{I_i}\chi_i(x)g_i(x)\,dt=g_i(x)$, the time--1 map produced by $\chi_ig_i$ on the interval $I_i$ equals $\phi^1_{g_i}$. Moreover, the time supports are disjoint so we have
\[\phi^1_{\widetilde{G}}=\prod_{i=1}^N\phi^1_{g_i}.\]
The generating Hamiltonian for the flow $(\phi^t_{\widetilde{G}})^{-1}\circ\phi^t_G$ is
\[K(t,x)=G(t,\phi^t_{\widetilde{G}}(x))-\widetilde{G}(t,\phi^t_{\widetilde{G}}(x)).\]
For $t\in I_i$ we have
\begin{equation}\label{ineq:one}
    \big\lvert G(t,x)-\widetilde{G}(t,x)\big\rvert=\big\lvert G(t,x)-\chi_i(t)\,g_i(x)\big\rvert\leq\big\lvert G(t,x)-Ng_i(x)\big\rvert+\big\lvert(N-\chi_i(t))g_i(x)\big\rvert.
\end{equation}
Set $C:=\sup_{(t,x)\in[0,1]\times M}|\partial_t G(t,x)|$. For all $t\in I_i$ we have
\begin{equation}\label{ineq:two}
    \big\lvert G(t,x)-Ng_i(x)\big\rvert=\Big\lvert N\int_{I_i}(G(t,x)-G(s,x))\,ds\Big\rvert\leq N\int_{I_i}C\cdot|t-s|ds\leq\frac{C}{N}.
\end{equation}
Set $C' := \max_{t \in [0,1]} \|G(t,\cdot)\|_{\infty}$. For all $t\in I_i$ we have
\begin{equation}\label{ineq:three}
    \big\lvert(N-\chi_i(t))\,g_i(x)\big\rvert= \Big\lvert N\int_{I_i}G(t,x)\Big\rvert\cdot\Big\lvert 1-\frac{1}{N}\,\chi_i(t)\Big\rvert\leq C'\cdot\Big\lvert 1-\frac{1}{N}\,\chi_i(t)\Big\rvert.
\end{equation}
Combining (\ref{Equation:Bump_function_condition}),(\ref{ineq:one}),(\ref{ineq:two}) and (\ref{ineq:three}) we get
\begin{equation}\label{ineq:L^(1,infty)bound}
    \|K\|_{L^{(1,\infty)}}=\int_{0}^{1}\|G(t,\cdot)-\widetilde{G}(t,\cdot)\|_{\infty}\,dt\leq\sum_{i=1}^{N}\int_{I_i}\left(\frac{C}{N}+C'\,\Big\lvert 1-\frac{\chi_i(t)}{N}\Big\rvert\right)\,dt\leq\frac{C+C'}{N}.
\end{equation}
    
\noindent Let $V \subset M$ be a subset of finite volume such that $\mathrm{supp}\,G(t,\cdot) \subset V$ for all $t \in [0,1]$. Then we also have $\mathrm{supp}\, K(t,\cdot)\subset V$, hence we obtain
\begin{equation}\label{ineq:L^(1,1)bound}
    \|K\|_{L^{(1,1)}}\leq\mathrm{Vol}(V)\cdot\|K\|_{L^{(1,\infty)}}\leq\frac{\mathrm{Vol}(V)\,(C+C')}{N}.
\end{equation}
Putting together (\ref{ineq:L^(1,infty)bound}) and (\ref{ineq:L^(1,1)bound}) and defining $c:=(\mathrm{Vol}(V)+1)(C+C')$ we get
\begin{equation}\label{Eq:Error_Bound}
    \vertiii{\phi^1_K}=\vertiii{(\phi^1_{\widetilde{G}})^{-1} \circ \phi^1_G}\leq\|K\|_{L^{(1,\infty)}}+\|K\|_{L^{(1,1)}}\leq\frac{c}{N}.
\end{equation}

\noindent The bound (\ref{ineq:two}) implies that for each $t\in I_i$ we have $|g_i(x)|<\frac{1}{N}\cdot|G(t,x)|+\frac{C}{N^2}$. In particular, $\|g_i\|_{\infty}\leq \frac{1}{N}\,\|G(t,\cdot)\|_{\infty}+\frac{C}{N^2}$ and therefore
\begin{equation}\label{Eq:Autonomous_Sum_Bound}
    \sum_{i=1}^N\|g_{i}\|_{\infty}=\sum_{i=1}^{N}\int_{I_i}N\,\|g_i\|_{\infty}\,dt\leq\sum_{i=1}^N\int_{I_{i}}\left(\|G(t,\cdot)\|_{\infty}+\frac{C}{N}\right)\,dt=\|G\|_{L^{(1,\infty)}}+\frac{C}{N}.
\end{equation}

\noindent Finally, combining (\ref{Eq:Error_Bound}), (\ref{Eq:Autonomous_Sum_Bound}), and Proposition~\ref{Prop:Autonomous_Bound}, we obtain
\[
\begin{aligned}
\vertiii{\phi^1_G}
&\leq \vertiii{\phi^1_{\widetilde{G}}} + \vertiii{(\phi^1_{\widetilde{G}})^{-1}\circ \phi^1_G}
= \vertiii{\prod_{i=1}^N \phi^1_{g_i}} + \vertiii{\phi^1_K} \\
&\leq \vertiii{\phi^1_K} + \sum_{i=1}^N \vertiii{\phi^1_{g_i}}
\leq \frac{c}{N} + \sum_{i=1}^{N}\|g_i\|_{\infty}
\leq \frac{c+C}{N} + \|G\|_{L^{(1,\infty)}}.
\end{aligned}
\]
where $C''=c+C$ depends only on $G$. By taking $N$ large enough we get the result.
\end{proof}

\begin{prop}[Autonomous Hamiltonian case]\label{Prop:Autonomous_Bound}
    Let $H\in C^{\infty}_{0,c}(M)$ be a zero-mean normalized autonomous Hamiltonian function. Then $\vertiii{\phi^1_H}\leq \|H\|_{\infty}$.
\end{prop}

\begin{proof}
    Pick $\varepsilon > 0$ and apply Proposition \ref{Prop:Approximation_by_Step_function} to obtain a function $K \in C_c^{\infty}(M)$ with the listed properties. Then
    \begin{align*}
        \vertiii{\phi^1_H}
        &= \vertiii{\phi^1_K \circ (\phi^1_K)^{-1} \circ \phi^1_H}
        \leq \vertiii{\phi^1_K} + \vertiii{(\phi^1_K)^{-1} \circ \phi^1_H} \\
        &= \vertiii{\phi^1_K} + \vertiii{\phi^1_{\overline{K} \# H}}
        \leq \vertiii{\phi^1_K} + \|\overline{K} \# H\|_{L^{(1,\infty)}} + \|\overline{K} \# H\|_{L^{(1,1)}} \\
        &= \vertiii{\phi^1_K} + \|H-K\|_{\infty} + \int_{M} |H-K| \,\omega^n
        \leq \|H\|_{\infty} + (c+2)\,\varepsilon,
    \end{align*}
    where the constants $c$ depends on $H$. Letting $\varepsilon \to 0$ yields the desired inequality.
\end{proof}

\begin{prop}\label{Prop:Approximation_by_Step_function}
    Let $\varepsilon>0$ and $H\in C^{\infty}_{0,c}(M)$. There exists $K\in C^{\infty}_c(M)$ such that
    \begin{enumerate}
        \item $\|H-K\|_{\infty}\leq \|H\|_{\infty}+\varepsilon$,
        \item $\int_{M}|H-K|\,\omega^n<\varepsilon$,
        \item $\vertiii{\phi^1_K}<c\cdot\varepsilon$,
    \end{enumerate}
    where $c>0$ is a constant that depends only on $H$ (and not on $K$).
\end{prop}

\subsection{Proof of Proposition \ref{Prop:Approximation_by_Step_function}}

\noindent Let $\delta>0$ be sufficiently small, and let $\mathcal{U}\subset M$ be a bounded and connected open subset such that $\mathrm{supp}\,H\subset \mathcal{U}$.
    
\begin{claim}\label{Claim:Setup_for_Sikorav_trick}
    There exists an integer $N=N(\delta)\in\mathbb{N}$ and three finite families of pairwise disjoint open subsets of $M$, namely
    \[
    \mathcal{Q}=\{Q_0,Q_1,\dots,Q_N\},\qquad
    \mathcal{R}=\{R_0,R_1,\dots,R_{L}\},\qquad
    \mathcal{R}'=\{R'_1,R'_2,\dots,R'_{L'}\},
    \]
    where $L=\big\lfloor\frac{N}{2}\big\rfloor$ and $L'=\big\lfloor\frac{N-1}{2}\big\rfloor$, such that the following properties hold:
    \begin{enumerate}[label=(\roman*)]
        \item For every $V \in \mathcal{Q} \cup \mathcal{R} \cup \mathcal{R}'$, the closure $\overline{V}$ is contained in $\mathcal{U}$ and is homeomorphic to the standard closed Euclidean ball.
        \item The diameter of each set in $\mathcal{Q}$ is at most $\delta$.
        \item $(\forall\,0\leq i\leq N)$ there exists a symplectic diffeomorphism $\phi_i:Q_0\to Q_i$.
        \item $(\forall\, 0\leq i\leq L)\,\overline{(Q_{2i}\cup Q_{2i+1})}\subset R_i$ and there exists a Hamiltonian diffeomorphism $\Psi_i\in\mathrm{Ham}_c(R_i)$ such that
        \[\Psi_i\circ\phi_{2i}=\phi_{2i+1},\]
        and $\Psi_i$ is generated by a normalized Hamiltonian function compactly supported in $R_i$ of $L^{(1,\infty)}$-norm less than $\delta$.
        \item $(\forall\, 1\leq i\leq L)\,\overline{(Q_{2i-1}\cup Q_{2i})}\subset R'_i$ and there exists a Hamiltonian diffeomorphism $\Psi'_i\in\mathrm{Ham}_c(R'_i)$ such that
        \[\Psi'_i\circ\phi_{2i-1}=\phi_{2i},\]
        and $\Psi'_i$ is generated by a normalized Hamiltonian function compactly supported in $R'_i$ of $L^{(1,\infty)}$-norm less than $\delta$.
        \item As $\delta \to 0$, the disjoint union of the sets in $\mathcal{Q}$ fills up the volume of $\mathcal{U}$.
    \end{enumerate}
\end{claim}
    
\noindent Let $F_0\in C_c^{\infty}(Q_0)$ be a function satisfying
\[0\leq F_0\leq 1+\delta,\quad\text{and}\quad\int_{Q_0} F_0 \,\omega^n = \mathrm{Vol}(Q_0).\]
For each $0\leq i\leq N$ define the real number $c_i$ and pick a point $a_i\in Q_i$ by
\[c_i = \frac{1}{\mathrm{Vol}(Q_i)}\int_{Q_i} H \,\omega^n = H(a_i).\]
The existence of $a_i\in Q_i$ follows from the continuity of $H$ and the intermediate value theorem. Finally, we define the function $K$ by
\begin{equation}\label{Equation:step_function_K}
    K = \sum_{i=0}^N c_i\,F_i=\sum_{i=0}^N c_i \cdot (F_0 \circ \phi_i^{-1}),
\end{equation}
where $F_i := F_0 \circ \phi_i^{-1}$. Each $F_i$ is supported in $Q_i$, hence $\mathrm{supp}(K) \subset \bigsqcup_{i=0}^N Q_i$.

\begin{claim}
    If $\delta>0$ is sufficiently small, then 
    \[\|H-K\|_{\infty} \leq \|H\|_{\infty} + \varepsilon,\qquad\int_{M} |H-K| \,\omega^n < \varepsilon.\]
\end{claim}

\begin{proof}
    Let $x \in Q_i$, and let $\gamma:[0,1]\to M$ be a smooth path of length at most $\delta$ connecting $a_i \in Q_i$ to $x$ (such a path exists because the diameter of $Q_i$ is bounded by $\delta$). Define
    \[
        C := \sup_{y\in\mathcal{U}} |dH(y)| < +\infty,
    \]
    where $|dH(y)|$ denotes the operator norm of the covector $dH(y)$. The finiteness of $C$ follows from the fact that $H$ has compact support. Then
    \begin{equation}\label{Eq:L_infty_Bound_H}
        |H(x)-c_i|=|H(\gamma(1))-H(\gamma(0))|=\Big\lvert\int_{0}^1 dH(\gamma'(t))\,dt\Big\rvert< C\cdot\mathrm{length}(\gamma)\leq C\delta.
    \end{equation}
    On the other hand, since $0 \leq F_0 \leq 1+\delta$, it follows that for each $x \in Q_i$ we have
    \begin{equation}\label{Eq:L_infty_Bound_F}
    \begin{cases}
        0 \leq c_i\,F_i(x) \leq c_i + \delta c_i, & \text{if } c_i \geq 0, \\[4pt]
        c_i + \delta c_i \leq c_i \, F_i(x) \leq 0, & \text{if } c_i < 0.
    \end{cases}
    \end{equation}
    \noindent Combining (\ref{Eq:L_infty_Bound_H}) and (\ref{Eq:L_infty_Bound_F}), we conclude that for each $x \in Q_i$ we have
    \begin{equation}\label{Eq:L_infty_Bound_(H-K)}
        \begin{cases}
            -(c_i+C)\delta\leq  H(x)-c_i\,F_i(x) \leq c_i+\delta C, & \text{if } c_i \geq 0, \\[4pt]
            c_i - \delta C \leq H(x)-c_i\,F_i(x) \leq \delta(C-c_i), & \text{if } c_i < 0.
        \end{cases}
    \end{equation}
    \noindent By choosing $\delta>0$ sufficiently small, equations (\ref{Eq:L_infty_Bound_H}) and (\ref{Eq:L_infty_Bound_(H-K)}) imply that
    \[|H(x)-K(x)| < |H(x)| + \varepsilon,\]
    as desired. For the second bound, we define 
    \[V=\mathcal{U}\setminus\bigsqcup_{i=0}^N Q_i,\]
    and note that $\mathrm{Vol}(V)\to 0$ as $\delta\to 0$. Now we can write
    \[\int_{M}|H-K|\,\omega^n=\int_{V}|H|\,\omega^n+\sum_{i=0}^N\int_{Q_i}|H-c_i\,F_i|\,\omega^n\]
    \[\leq\int_{V}|H|\,\omega^n+\sum_{i=0}^N\int_{Q_i}|H-c_i(1+\delta)|\,\omega^n+\sum_{i=0}^N|c_i|\cdot\int_{Q_i}(1+\delta-F_i)\,\omega^n\]
    We now bound each summand. Let $C' := \max |H| < +\infty$. Then:  
    \begin{itemize}
        \item $\displaystyle \int_{V} |H| \,\omega^n \leq \mathrm{Vol}(V)\cdot C'$, which approaches $0$ as $\delta \to 0$.
        \item From (\ref{Eq:L_infty_Bound_H}) we have
        \[
        -(C+c_i)\delta \leq H(x) - c_i(1+\delta) \leq \delta(C-c_i).
        \]
        Since $|c_i| = |H(a_i)| \leq C'$, it follows that
        \[
        |H - c_i(1+\delta)| \leq (C+C')\delta,
        \]
        and therefore
        \[
        \int_{Q_i} |H - c_i(1+\delta)| \,\omega^n
        \leq \delta(C+C') \cdot \mathrm{Vol}(Q_i).
        \]
        \item Since $\int_{Q_i} F_i \,\omega^n = \mathrm{Vol}(Q_i)$ and $|c_i| < C'$, we obtain
        \[
        |c_i| \int_{Q_i} (1+\delta - F_i)\,\omega^n
        \leq \delta C' \cdot \mathrm{Vol}(Q_i).
        \]
    \end{itemize}
    Combining these estimates yields
    \[
        \int_{M} |H-K| \,\omega^n
        \leq \mathrm{Vol}(V)\cdot C' + \delta \cdot \mathrm{Vol}(\mathcal{U})(C+2C'),
    \]
    which tends to $0$ as $\delta \to 0$.
\end{proof}
    
\noindent It remains to prove that $\vertiii{\phi^1_K}\leq\varepsilon$, where $c>0$ is a constant that depends only on $H$. The following claim is essentially due to Sikorav (see Section 8.4 in \cite{Si90}); however, the proof presented here is almost entirely adapted from Lemma 2.1 in \cite{Bu23}.

\begin{claim}\label{Claim:Sikorav_trick}
    If $\delta>0$ is small enough, the Hamiltonian diffeomorphism $\phi^1_K$ (where $K$ is defined by (\ref{Equation:step_function_K})) can be generated by an autonomous Hamiltonian $K'$ supported in $\mathcal{U}$ such that $\|K'\|_{L^{(1,\infty)}}<\varepsilon$.
\end{claim}

\noindent Before state the proof of the claim, let us see how to use it to finish the proof. Note that
\[\vertiii{\phi^1_K}=\vertiii{\phi^1_{K'}}\leq\|K'\|_{L^{(1,\infty)}}+\|K'\|_{L^(1,1)}\leq\|K'\|_{L^{(1,\infty)}}(1+\mathrm{Vol}(\mathcal{U})),\]
which completes the proof of Proposition \ref{Prop:Approximation_by_Step_function}.\qed

\subsection{Proof of Claim \ref{Claim:Sikorav_trick}}
    We restrict to the open symplectic submanifold $\mathcal{U}\subset M$, and all Hamiltonian diffeomorphisms considered in the proof of this claim are assumed to have compact support in $\mathcal{U}$.\\
    
    For each $0 \leq i \leq N$, define $\mathfrak{f}_i \in \mathrm{Ham}_c(\mathcal{U},\omega)$, supported in $Q_i$, as the time--$1$ map of the Hamiltonian isotopy generated by the Hamiltonian function $F_i$:
    \[\mathfrak{f}_i := \phi^1_{F_i}.\]
    
    \noindent Define Hamiltonian diffeomorphisms $\Phi,\Phi'\in\mathrm{Ham}_c(\mathcal{U},\omega)$ as 
    \[\Phi:=\phi^1_K=\prod_{i=1}^N\mathfrak{f}_i,\quad\Phi':=\mathfrak{f}_0\,\prod_{i=1}^N(\phi_i^{-1}\,\mathfrak{f}_i\,\phi_i),\]
    
    \noindent where $\phi_1,\ldots,\phi_N$ are Hamiltonian diffeomorphisms defined in Claim \ref{Claim:Setup_for_Sikorav_trick}. The Hamiltonian diffeomorphism $\Phi'$ is generated by an autonomous Hamiltonian function
    \[\widetilde{K}=\left(\sum_{i=1}^{N}c_i\right)\cdot F_0.\]
    Note that $\|F_0\|_{\infty}\leq 1+\delta$. Denote $V=\mathcal{U}\setminus\bigsqcup_{Q\in\mathcal{Q}}Q$. The property \textit{(vi)} in Claim \ref{Claim:Setup_for_Sikorav_trick} implies that the $\mathrm{Vol}(V)\to 0$ as $\delta\to 0$. Therefore, for $\delta>0$ small enough, we have
    \begin{align*}
        \|\widetilde{K}\|_{\infty} 
        &\leq (1+\delta)\cdot\Big\lvert\sum_{i=1}^{N}c_i\Big\rvert=(1+\delta)\cdot\Big\lvert\int_{\bigsqcup_{Q\in\mathcal{Q}}Q} H\,\omega^n\Big\rvert \\
        &= (1+\delta)\cdot\Big\lvert\int_{V} H\,\omega^n\Big\rvert\leq (1+\delta)(\max |H|)\cdot\mathrm{Vol}(V) < \frac{\varepsilon}{2}.
    \end{align*}
    \noindent This in particular implies that $\|\Phi'\|_{\mathrm{Hofer}}\leq\frac{\varepsilon}{2}$.\\
        
    \noindent Define Hamiltonian diffeomorphisms $\Psi,\Psi'\in\mathrm{Ham}_c(\mathcal{U},\omega)$ as
    \[\Psi:=\prod_{i=0}^L\Psi_i,\quad\Psi':=\prod_{i=1}^L\Psi'_i,\]
    where $\Psi_i,\Psi'_i$ are Hamiltonian diffeomorphisms defined in Claim \ref{Claim:Setup_for_Sikorav_trick}. Additionally, we introduce $\mathfrak{g}_1, \mathfrak{g}_2, \ldots, \mathfrak{g}_L \in \mathrm{Ham}_c(\mathcal{U},\omega)$:
    \[
    \mathfrak{g}_i :=
    \begin{cases}
    \mathfrak{f}_{2i}\,\Psi^{-1}\,\mathfrak{f}_{2i+1}\,\Psi
    = \mathfrak{f}_{2i} \,(\phi_{2i+1}\phi_{2i}^{-1})\,\mathfrak{f}_{2i+1}, 
    & \parbox[t]{0.5\textwidth}{if $N=2L+1$ is odd, or\\
    if $N=2L$ is even and $0 \leq i \leq L-1$}, \\[1mm]
    \mathfrak{f}_{2L} = \mathfrak{f}_N, 
    & \text{if $N=2L$ is even and $i=L$}.
    \end{cases}
    \]
    Note that $\mathrm{supp}(\mathfrak{g}_i)\subset Q_{2i}$ for $0\leq i\leq L$. Denote $\widetilde{\Phi}:=\mathfrak{g}_0\,\mathfrak{g}_1\cdots\mathfrak{g}_L$. Then we have
    \[\Phi^{-1}\widetilde{\Phi}=\prod_{i=0}^{L'}(\Psi^{-1}\,\mathfrak{f}_{2i+1}\,\Psi\,\mathfrak{f}_{2i+1}^{-1})=\left(\prod_{i=0}^{L'}\mathfrak{f}_{2i+1}\right)^{-1}\Psi^{-1}\left(\prod_{i=0}^{L'}\mathfrak{f}_{2i+1}\right)\,\Psi,\]
    and hence
    \begin{equation}\label{Eq:Sikorav_1}
        d_{\mathrm{Hofer}}(\Phi,\widetilde{\Phi})=\|\Phi^{-1}\,\widetilde{\Phi}\|_{\mathrm{Hofer}}=2\|\Psi\|_{\mathrm{Hofer}}\leq2\delta.
    \end{equation}
    For each $0\leq i\leq L$ define 
    \[\hat{\mathfrak{g}}_i:=\phi_{2i}^*\,\mathfrak{g}_i,\quad\text{and}\quad\hat{\mathfrak{h}}_i:=\prod_{j=i}^L\hat{\mathfrak{g}}_j\]
    Moreover, we define $\mathfrak{h}_{2i}:=(\phi_{2i})_*\,\hat{\mathfrak{h}}_i$ for $0\leq i\leq L$ and $\mathfrak{h}_{2i-1}:=(\phi_{2i-1})_*\,\hat{\mathfrak{h}}_i^{-1}$ for $1\leq i\leq L$. Define $\widehat{\Phi}:=\mathfrak{h}_0\,\mathfrak{h}_1\cdots\mathfrak{h}_{2L}$. Then
    \[\widetilde{\Phi}^{-1}\,\widehat{\Phi}=\left(\prod_{i=1}^{L}\mathfrak{h}_{2i-1}\right)^{-1}\Psi^{-1}\left(\prod_{i=1}^{L}\mathfrak{h}_{2i-1}\right)\,\Psi,\]
    and hence
    \begin{equation}\label{Eq:Sikorav_2}
        d_{\mathrm{Hofer}}(\widetilde{\Phi},\widehat{\Phi})=\|\widetilde{\Phi}^{-1}\,\widehat{\Phi}\|_{\mathrm{Hofer}}=2\|\Psi\|_{\mathrm{Hofer}}\leq2\delta.
    \end{equation}
    \noindent Finally, note that
    \[\mathfrak{h}_0=\hat{\mathfrak{h}}_0=\prod_{j=0}^L\hat{\mathfrak{g}}_j=\prod_{i=0}^{N}\phi_i^*\,\mathfrak{f}_i=\Phi',\]
    therefore
    \[(\Phi')^{-1}\widehat{\Phi}=\mathfrak{h}_0^{-1}\,\widehat{\Phi}=\prod_{i=1}^{2L}\mathfrak{h}_i=\left(\prod_{i=1}^{L}\mathfrak{h}_{2i-1}\right)^{-1}\Psi'\left(\prod_{i=1}^{L}\mathfrak{h}_{2i-1}\right)\,(\Psi')^{-1},\]
    and hence we get
    \begin{equation}\label{Eq:Sikorav_3}
        d_{\mathrm{Hofer}}(\widehat{\Phi},\Phi')=\|(\Phi')^{-1}\,\widehat{\Phi}\|_{\mathrm{Hofer}}=2\|\Psi'\|_{\mathrm{Hofer}}\leq2\delta.
    \end{equation}
    
    \noindent The inequalities (\ref{Eq:Sikorav_1}),(\ref{Eq:Sikorav_2}) and (\ref{Eq:Sikorav_3}) imply that $d_{\mathrm{Hofer}}(\Phi,\Phi')\leq6\delta$. Lastly, by picking $\delta>0$ small enough, we get that $\|\Phi\|_{\mathrm{Hofer}}\leq d_{\mathrm{Hofer}}(\Phi,\Phi')+\|\Phi'\|_{\mathrm{Hofer}}\leq 2\varepsilon/3$, which is enough to complete the proof.

\subsection{Proof of Claim \ref{Claim:Setup_for_Sikorav_trick}}

\noindent Let $\{(U_{i},\varphi_i)\}_{i=1}^{m}$ be a finite family of Darboux balls such that $\bigcup_{i=1}^m U_i=\mathcal{U}\supset\mathrm{supp}\,f$. After applying Lemma~\ref{Lemma:transition_charts_lemma}, we may assume that for every $1\leq i,j\leq m$ with $U_{i}\cap U_j\neq \emptyset$ there exists an open subset $U_{ij}\subset U_i\cap U_{j}$ on which the (possibly modified) transition map from $U_i$ to $U_{j}$ restricts to the identity. Let $\{V_i\}_{i=1}^m$ be a family of subsets of $M$ defined as $V_1=U_1$ and $V_i=U_i\setminus\bigcup_{j=1}^{i-1}U_i$ for $1<i\leq m$. Note that $\bigsqcup_{i=1}^m V_i=\mathcal{U}$.\\

\noindent Fix $\varepsilon > 0$. Let $\{\mathcal{Q}_i\}_{i=1}^m$ be families of disjoint open sets such that each element of $\mathcal{Q}_i$ 
is a standard cube of side length $a > 0$ contained in the chart $\varphi_i(V_i) \subset \varphi_i(U_i) \subset \mathbb{R}^{2n}$, and
\[
\mathrm{Vol}\Big( \bigsqcup_{i=1}^m \bigsqcup_{Q \in \mathcal{Q}_i} Q \Big)
> \mathrm{Vol}(\mathcal{U}) - \varepsilon.
\]
Let $\mathcal{G} = (\mathcal{V}, \mathcal{E})$ be the graph whose vertex set is 
$\mathcal{V} := \bigsqcup_{i=1}^m \bigsqcup_{Q \in \mathcal{Q}_i} Q$. 
Let $Q, Q' \in \mathcal{V}$ be two vertices, where $Q \in V_i$ and $Q' \in V_j$ for some $1 \le i, j \le m$. 
We place an edge between $Q$ and $Q'$ if and only if one of the following holds:
\begin{itemize}
    \item $i = j$;
    \item $i \ne j$, $U_i \cap U_j \ne \emptyset$, and one of the cubes $Q, Q'$ belongs to $U_{ij}$.
\end{itemize}

\noindent Since $\mathcal{U}$ is connected, the graph $\mathcal{G}$ is also connected. Moreover, if $\varepsilon > 0$, and consequently $a > 0$, are chosen sufficiently small, the vertices of the graph $\mathcal{G}$ can be ordered in a sequence $\widetilde{Q}_1, \widetilde{Q}_2, \ldots, \widetilde{Q}_{|\mathcal{V}|}$ such that there is an edge between every two consecutive vertices. For each $1 \leq i < |\mathcal{V}|$, let $\gamma_i : [0,1] \to M$ be a smoothly embedded curve satisfying:
\begin{enumerate}
    \item $\gamma_i(0)$ is a vertex of the cube $\widetilde{Q}_i$ (opposite to $\gamma_{i-1}(0)$ if $i > 1$), 
    and $\gamma_i(1)$ is a vertex of the cube $\widetilde{Q}_{i+1}$,
    \item $\gamma_i([0,1]) \subset V_k$ if $\widetilde{Q}_i,\widetilde{Q}_{i+1}\in V_k$, and otherwise $\gamma_i([0,1])\subset U_k$ if $\widetilde{Q}_i,\widetilde{Q}_{i+1}\in U_k$,
    \item $\gamma_i((0,1)) \cap \bigcup_{j=1}^{|\mathcal{V}|} \widetilde{Q}_j = \emptyset$ and $\gamma_i([0,1])\cap\gamma_j([0,1])=\emptyset$ whenever $i\neq j$.
\end{enumerate}

\noindent Note that for every $1 \leq i \leq m$, the set 
$U_i \setminus \bigcup_{j=1}^{|\mathcal{V}|} \widetilde{Q}_j$ is connected. Moreover, since every edge of the graph $\mathcal{G}$ is contained in some $U_k$, 
such curves $\gamma_i$ can indeed be constructed.

\begin{figure}[h]
    \centering
    \includegraphics[scale=0.9]{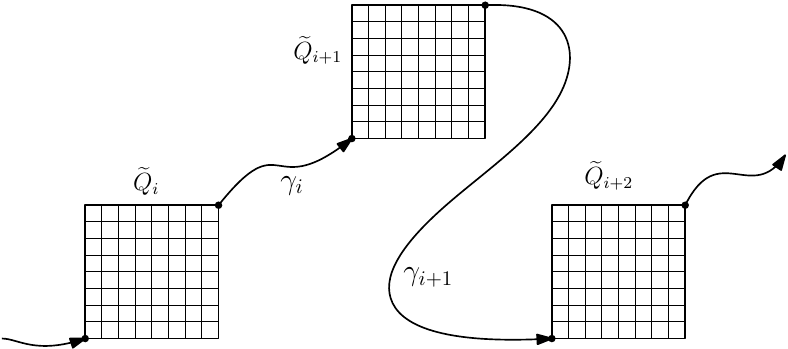}
    \caption{Cubes $\widetilde{Q}_i$ subdivided into smaller cubes}
    \label{fig:cubes}
\end{figure}

\noindent Subdivide each cube $\widetilde{Q}_i$ into smaller cubes of diameter less than $\delta$, then slightly shrink each of these cubes and label them as $Q_i^1, Q_i^2, \ldots, Q_i^K$ (see Figure \ref{fig:cubes}).  
Assume that the cubes satisfy the following properties:
\begin{itemize}
    \item The total volume condition:
    \[
    \sum_{i=1}^{|\mathcal{V}|} \sum_{j=1}^{K} \mathrm{Vol}(Q_i^j) > \mathrm{Vol}(\mathcal{U}) - \varepsilon,
    \]
    \item $Q_i^1$ touches the corner $\gamma_{i-1}(1)$ of $\widetilde{Q}_i$, and $Q_i^K$ touches the corner $\gamma_i(0)$ of $\widetilde{Q}_i$,
    \item For each $1 \leq j < K$, the cubes $Q_i^j$ and $Q_i^{j+1}$ shared a common side before shrinking.
\end{itemize}

\noindent Finally, define $\mathcal{Q} := \{ Q_i^j \mid 1 \le i \le |\mathcal{V}|, \, 1 \le j \le K \}$, and order its elements by declaring
\[
Q_{i_1}^{j_1} \prec Q_{i_2}^{j_2} \quad \text{if either } i_1 < i_2, \text{ or } i_1 = i_2 \text{ and } j_1 < j_2.
\]
Denote $\mathcal{Q} = \{Q_1, \ldots, Q_N\}$, with the indices ordered according to the previously defined order, and $N = K \cdot |\mathcal{V}|$. Let us define the family $\mathcal{R}$. For each $i$, consider the pair of cubes $Q_{2i}, Q_{2i+1}$. There are two cases:
\begin{itemize}
    \item If both cubes belong to the same cube $\widetilde{Q}_j$ for some $1 \le j \le |\mathcal{V}|$, then $Q_{2i}$ and $Q_{2i+1}$ shared a common edge before shrinking. In this case, we can define $R_i$ to be a rectangle containing both cubes.
    
    \item Otherwise, $Q_{2i} \subset \widetilde{Q}_k$ and $Q_{2i+1} \subset \widetilde{Q}_{k+1}$ for some $1 \le k < |\mathcal{V}|$. Moreover, $Q_{2i}$ touches a corner of $\widetilde{Q}_k$ and $Q_{2i+1}$ touches a corner of $\widetilde{Q}_{k+1}$, and they are connected via the curve $\gamma_k$. Let $V_l$ and $V_{l'}$ be sets such that $\widetilde{Q}_k \subset V_l$ and $\widetilde{Q}_{k+1} \subset V_{l'}$. 

    \begin{itemize}
        \item If $l = l'$, the image of $\gamma_k$ lies entirely inside $V_l$.  
        \item Otherwise, $U_l \cap U_{l'} \neq \emptyset$, and at least one of the cubes $\widetilde{Q}_k, \widetilde{Q}_{k+1}$ belongs to $U_{ll'}$.
    \end{itemize}

    In either case, if $\delta$ is chosen sufficiently small, we can define $R_i$ to be a tubular neighborhood of $\mathrm{Im}\,\gamma_k$, 
    and map $Q_{2i}$ to $Q_{2i+1}$ via a Hamiltonian isotopy supported inside $R_i$, 
    which translates $Q_{2i}$ along the curve $\gamma_k$ all the way to $Q_{2i+1}$. It is a well-known fact that this can be achieved by a Hamiltonian isotopy whose Hofer norm is as close as we want to the displacement energy of $Q_{2i}$ with is less than $\delta$.
\end{itemize}
\noindent We use the same construction for $\mathcal{R}'$, and with it we complete the proof.\qed

\section{Proof of Theorem \ref{Theorem:Classification_of_bi-invariant_metrics}}

\noindent\textbf{\underline{Case 1:}} \textit{There does not exist a constant $c > 0$ such that $\|f\| \geq c \|f\|_{\infty}$ for all $f \in C^{\infty}_c(M)$.}

\smallskip

This condition is equivalent to: for any $\varepsilon > 0$, there exists $f \in C_c^{\infty}(M)$ with $\|f\| \le \varepsilon$ and $\|f\|_{\infty} = 1$. 
Let $\phi \in \mathrm{Ham}_c(M,\omega)$ satisfy $\phi(p) \notin \mathrm{supp}\,f$ for some $p \in M$ with $|f(p)| = 1$. Then $g := \phi^*f - f \in C_{0,c}^{\infty}(M)$ satisfies $\|g\| \le 2\varepsilon$ and $1 \le \|g\|_{\infty} \le 2$. Thus, no constant $c > 0$ exists such that $\|g\| \ge c\,\|g\|_{\infty}$ for all $g \in C_{0,c}^{\infty}(M)$. By Theorem~\ref{Theorem:Lower_Bound_Open_Case}, the pseudo-distance $\rho$ is degenerate, and by the Eliashberg–Polterovich classification (Theorem~1.4.A in \cite{EP93}), $\rho$ is equivalent to $\mu\,|\mathrm{Cal}|$ for some $\mu \ge 0$.

\smallskip

\noindent\textbf{\underline{Case 2:}} \textit{There exists a constant $c>0$ such that $\|f\|\geq c\|f\|_{\infty}$ for all $f\in C^{\infty}_c(M)$.}

\smallskip

Corollary \ref{Corollary:upper_bound}, together with our assumption, implies that there exists $C>0$ such that for all $f\in C^{\infty}_c(M)$ we have
\begin{equation}
    c\|f\|_{\infty}\leq\|f\|\leq C(\|f\|_{\infty}+\|f\|_{L^1}).
\end{equation}

\noindent\underline{\textbf{Case 2.1:} $\mathrm{Vol}(M)<\infty$.}

\smallskip

Then $c\|\cdot\|_{\infty}\leq\|\cdot\|\leq C(1+\mathrm{Vol}(M))\|\cdot\|_{\infty}$, so $\rho$ is equivalent to Hofer’s metric.

\smallskip

\noindent\underline{\textbf{Case 2.2:} $\mathrm{Vol}(M) = \infty$.}

\smallskip

Let $\{h_k\}_{k=1}^{\infty} \subset C^{\infty}_c(M)$ be a sequence satisfying
\begin{equation}\label{Equation:classification}
    0\leq h_k\leq\tfrac{1}{k}, \quad \int_{M} h_k \,\omega^n = 1,\quad\mathrm{Vol}(\{h_k=\tfrac{1}{k}\})>k-\tfrac{1}{k}.
\end{equation}
Such a sequence exists because $\mathrm{Vol}(M)=\infty$. Moreover, $\|h_k\|\leq C(\|h_k\|_{\infty}+\|h_k\|_{L^1})\leq 2C$, and hence there exists
$\liminf_{k\to\infty}\|h_k\|$. We apply Theorem \ref{TheoremOW:extension} to extend the norm $\|\cdot\|$ to a norm $\|\cdot\|'$ on the space $L^{\infty}_c(M)$. Let $W_k$ be a bounded measurable set with $\mathrm{Vol}(W_k)=k$. For each $k\in\mathbb{N}$, define a function
\[F_k:=\frac{1}{k}\cdot\mathbf{1}_{W_k}\in L^{\infty}_c(M).\]

\begin{claim}
The number $b := \liminf_{k \to \infty} \|h_k\|$ does not depend on the choice of the sequence $\{h_k\}_{k=1}^{\infty} \subset C^{\infty}_c(M)$ satisfying (\ref{Equation:classification}), and it coincides with $\liminf_{k \to \infty} \|F_k\|'$.
\end{claim}

\begin{proof}
    By passing to a converging subsequence if necessary, we may assume $\lim_{k\to\infty}\|h_k\|=b$. Let $\varphi_k:M\to M$ be a compactly supported volume-preserving bijection satisfying $\{h_k=\tfrac{1}{k}\}\subset\varphi_k( W_k)$. Then 
    \[\mathrm{Vol}(\{|F_k\circ\varphi_k-h_k|>\tfrac{1}{k}\})<\tfrac{1}{k},\]
    implying that the sequence $F_k\circ\varphi_k-h_k$ converges in measure to $0$. We can now use
    \[\big\lvert\|F_k\circ\varphi_k\|'-\|h_k\|'\big\rvert\leq \|F_k\circ\varphi_k-h_k\|'\xrightarrow{k\to\infty}0,\]
    to conclude $\lim_{k\to\infty}\|F_k\|'=\lim_{k\to\infty}\|F_k\circ\varphi_k\|'=\lim_{k\to\infty}\|h_k\|'=\lim_{k\to\infty}\|h_k\|=b$.
\end{proof}

\smallskip

Fix $\varphi \in \mathrm{Ham}(M,\omega)$ and let $H \in C^{\infty}_c([0,1]\times M)$ be a Hamiltonian with $\phi^1_H = \varphi$. Set $c(t) := \int_M H(t,\cdot)\,\omega^n$, and let $\{h_k\}_{k=1}^{\infty}$ satisfy (\ref{Equation:classification}) with $h_k|_{\mathrm{supp}\,H} \equiv \tfrac{1}{k}$ for $k$ large. Passing to a subsequence if needed, assume $\lim_{k\to\infty}\|h_k\|=b$. Define
\[
\widetilde{H}_k(t,x) := H(t,x) - c(t)h_k(x).
\]
Then $\int_M \widetilde{H}_k(t,\cdot)\,\omega^n = 0$ for all $t\in[0,1]$, hence $\phi^1_{\widetilde{H}_k} \in \ker(\mathrm{Cal})$. Using the upper bound $\|\cdot\|\leq C(\|\cdot\|_{\infty}+\|\cdot\|_{L^1})$ and Theorem \ref{Theorem:Upper_Bound_Infinite_volume}, we get
\[
\vertiii{\phi^{1}_{\widetilde{H}_k}} \leq C\|\widetilde{H}_k\|_{L^{(1,\infty)}}.
\]
Moreover, for $k$ large enough we have $\phi^1_H = \phi^1_{\widetilde{H}_k}\phi^{\mathrm{Cal}(\phi^1_H)}_{h_k}$, hence
\[
\vertiii{\varphi} = \vertiii{\phi^1_H} \leq \vertiii{\phi^1_{\widetilde{H}_k}} + \vertiii{\phi^{\mathrm{Cal}(\phi^1_H)}_{h_k}}
\leq C\|\widetilde{H}_k\|_{L^{(1,\infty)}} + \|h_k\|\cdot|\mathrm{Cal}(\phi^1_H)|.
\]
Taking $k\to\infty$ yields $\|\phi^1_H\| \leq C\|H\|_{L^{(1,\infty)}} + b\cdot|\mathrm{Cal}(\phi^{1}_H)|$. Minimizing over all $H$ generating $\varphi$, we obtain
\begin{equation}\label{Equation:upper_bound_for_invariant_norm}
    c\,\|\varphi\|_{\mathrm{Hofer}}\leq\vertiii{\varphi} \leq C\,\|\varphi\|_{\mathrm{Hofer}} + b\cdot|\mathrm{Cal}(\varphi)|,
\end{equation}

\noindent where we used the inequality $\|\cdot\|\ge c\|\cdot\|_{\infty}$ to obtain the lower bound.

\smallskip

\noindent\underline{\textbf{Case 2.2.a:} $b=0$.}

\smallskip

In this case $c\|\varphi\|_{\mathrm{Hofer}}\leq\vertiii{\varphi}\leq C\|\varphi\|_{\mathrm{Hofer}}$, so $\rho$ is equivalent to Hofer’s metric.

\smallskip 

\noindent\underline{\textbf{Case 2.2.b:} $b>0$.}

\smallskip

Let $H \in C^{\infty}_c([0,1] \times M)$ be a Hamiltonian generating $\varphi$, and set $H_t(x) = H(t,x)$ for $t \in [0,1]$. Apply Theorem \ref{TheoremOW:extension} to extend $\|\cdot\|$ to a norm $\|\cdot\|'$ on $L^{\infty}_c(M)$, and then use Lemma \ref{LemmaOW:partition} for $H_t$ to obtain
\begin{equation}\label{Eq:averaging_property}
    \frac{1}{\mathrm{Vol}(S)} \Big\lvert \int_{M} H_t \, \omega^n\Big\rvert\cdot\|\mathbf{1}_{S}\|' 
    = \|\langle H_t \rangle_{S} \mathbf{1}_{S}\|' 
    \leq \|H_t\|' = \|H_t\|,
\end{equation}
for every bounded measurable $S \supset \mathrm{supp}\,H_t$.  
Let $\{S_k\}_{k=1}^{\infty}$ be an increasing sequence of bounded measurable subsets of $M$ with $\bigcup_{t \in [0,1]} \mathrm{supp}\,H_t \subset S_k$ and $\lim_{k \to \infty} \mathrm{Vol}(S_k) = \infty$.  
Define $G_k := \frac{1}{\mathrm{Vol}(S_k)} \mathbf{1}_{S_k} \in L^{\infty}_c(M)$ and apply (\ref{Eq:averaging_property}) to obtain
\[
\int_{0}^{1} \|H_t\| \, dt \geq \int_{0}^{1} \big\|G_k\big\|' \cdot \Big\lvert \int_M H_t \, \omega^n \Big\rvert \, dt 
\geq \big\|G_k\big\|' \cdot \Big\lvert \int_{0}^{1} \int_M H_t \, \omega^n \, dt \Big\rvert 
= \big\|G_k\big\|' \cdot |\mathrm{Cal}(\varphi)|.
\]
Taking $\liminf_{k \to \infty}$, we obtain $\vertiii{\varphi} \ge b\,|\mathrm{Cal}(\varphi)|$. Together with $\vertiii{\varphi} \ge c\,\|\varphi\|_{\mathrm{Hofer}}$, this yields
\begin{equation}\label{Equation:lower_bound_for_invariant_norm}
    \vertiii{\varphi} \ge \tfrac{c}{2}\,\|\varphi\|_{\mathrm{Hofer}} + \tfrac{b}{2}\,|\mathrm{Cal}(\varphi)|.
\end{equation}
Combining (\ref{Equation:upper_bound_for_invariant_norm}) and (\ref{Equation:lower_bound_for_invariant_norm}), we conclude that $\rho$ is equivalent to $d_{\mathrm{Hofer}} + |\mathrm{Cal}|$.

\section{Appendix: Proof of Theorem \ref{Theorem:Lower_Bound_Open_Case}}

We follow the same approach as in \cite{OW05} and present arguments adapted to our setting.

\begin{thm}\label{TheoremOW:extension}
    Let $\|\cdot\|$ be $\mathrm{Ham}(M,\omega)$-invariant norm on the space $C^{\infty}_c(M)$ such that $\|\cdot\|\leq C(\|\cdot\|_{\infty}+\|\cdot\|_{L^1})$ for some constant $C>0$. Then $\|\cdot\|$ can be extended to a semi-norm $\|\cdot\|'\leq C(\|\cdot\|_{\infty}+\|\cdot\|_{L^1})$ on $L^{\infty}_c(M)$, which is invariant under all compactly supported measure preserving bijections on $M$.
\end{thm}

\begin{proof}
    Any function $F \in L^{\infty}_c(M)$ can be approximated in measure by smooth compactly supported functions. We then define
    \[
    \|F\|' := \inf \big\{ \liminf_{i \to \infty} \|F_i\| \big\},
    \]
    where the infimum is over all uniformly bounded sequences $\{F_i\}_{i=1}^\infty \subset C_c^\infty(M)$ with supports contained in a single compact set and converging to $F$ in measure. Since both the infimum and $\liminf$ respect scaling, $\|\cdot\|'$ is positively homogeneous. To check the triangle inequality, let $F,G \in L^{\infty}_c(M)$ and pick $\varepsilon$-approximating sequences $\{F_n\},\{G_n\}$ such that $\liminf_{n\to\infty}\|F_n\| \le \|F\|' + \varepsilon$ and $\liminf_{n\to\infty}\|G_n\| \le \|G\|' + \varepsilon$. Then $\liminf_{n\to\infty} \|F_n + G_n\| \le \|F\|' + \|G\|' + 2\varepsilon$, so $\|F+G\|' \le \|F\|' + \|G\|'$. Hence, $\|\cdot\|'$ is a semi-norm.
    \begin{claim}[Ostrover--Wagner]
        For every $F\in C^{\infty}_c(M)$ we have $\|F\|=\|F\|'$.
    \end{claim}
    \begin{proof}
        The inequality $\|F\|'\leq\|F\|$ follows immediately by taking sequence $F_i\equiv F$. It remains to prove that $\|F\|'\geq\|F\|$.
        Let $\{F_i\}_{i=1}^{\infty}$ be an uniformly bounded sequence of smooth functions converging in measure to $F$ and let $U\subset M$ be a bounded subset that contains $\mathrm{supp}\,F$ and $\mathrm{supp}\,F_i$ for all $i$. By restricting to $C^{\infty}_c(U)$, the condition $\|\cdot\|\leq C(\|\cdot\|_{\infty}+\|\cdot\|_{L^1})$ implies that $\|G\|\leq C'\|G\|_{\infty}$ for all $G\in C^{\infty}_c(U)$ and $C'=C(1+\mathrm{Vol}(U))$. We can now apply the same exact argument as the one in the proof Claim 3.1 \cite{OW05} to get a sequence $\{\widetilde{F}_i\}_{i=1}^{\infty}\subset C^{\infty}_c(U)$ such that $\|\widetilde{F}_i\|\leq\|F_i\|$ and $\lim_{i\to\infty}\|\widetilde{F}_i\|=\|F\|$. This in particular implies that $\|F\|=\lim_{i\to\infty}\|\widetilde{F}_i\|\leq\liminf_{i\to\infty}\|F_i\|$, and hence $\|F\|\leq\|F\|'$.
    \end{proof}
    
    \begin{claim}[Ostrover--Wagner]
        For every $F\in L^{\infty}_c(M)$ and every compactly supported measure preserving bijection $\varphi:M\to M$ we have $\|F\circ\varphi\|'=\|F\|'$.
    \end{claim}
    \begin{proof}
        See Claim 3.2 in \cite{OW05}.
    \end{proof}
    
    \noindent Finally, we prove that $\|\cdot\|'\leq C(\|\cdot\|_{\infty}+\|\cdot\|_{L^1})$. For simplicity we assume $M=\mathbb{R}^{2n}$, otherwise we can use partition of unity to reduce to this case. Pick $F\in L^{\infty}_c(\mathbb{R}^{2n})$. Then $F\in L^1(\mathbb{R}^{2n})$. Choose a standard family of mollifiers $\rho_{\varepsilon}\in C^{\infty}_c(\mathbb{R}^{2n})$ (with $\varepsilon>0$) satisfying $\rho_{\varepsilon}\geq 0$, $\int_{\mathbb{R}^{2n}}\rho_{\varepsilon}=1$ and $\mathrm{Vol}(\mathrm{supp}\,\rho_{\varepsilon})\xrightarrow{\varepsilon\to 0} 0$. Define $F_{\varepsilon}:=F\ast\rho_{\varepsilon}\in C^{\infty}_c(\mathbb{R}^{2n})$. One can check that $\|F_{\varepsilon}\|_{\infty}\leq \|F\|_{\infty}$ and Young's convolution inequality implies that $\|F_{\varepsilon}\|_{L^1}\leq\|F\|_{L^1}$, and therefore $\|F_{\varepsilon}\|\leq C(\|F_{\varepsilon}\|_{\infty}+\|F_{\varepsilon}\|_{L^1})\leq C(\|F\|_{\infty}+\|F\|_{L^1})$. Using the fact that $F_{\varepsilon}\xrightarrow{\varepsilon\to 0} F$ in measure, we get the desired inequality.
\end{proof}

\begin{lemma}[Ostrover--Wagner]\label{LemmaOW:partition}
    Let $F\in C_c(M)$, and let $S_1,\ldots,S_k$ be bounded finite measure sets with $\mathrm{supp}\,F\subset S_1\sqcup\ldots\sqcup S_k$. Then
    \[\|\langle F\rangle_{S_1}\mathbf{1}_{S_1}+\ldots+\langle F\rangle_{S_k}\mathbf{1}_{S_k}\|'\leq\|F\|',\] where $\langle F\rangle_S:=\frac{1}{\mathrm{Vol}(S_i)}\int_{S}F\,\omega^n$.
\end{lemma}

\begin{proof}
    Since $F$ has a compact support, $\|F\|\leq C\,(\|F\|_{\infty}+\|F\|_{L^1})$ implies that $\|F\|\leq C'\,\|F\|_{\infty}$ for $C'=C\,(1+\mathrm{Vol}(\mathrm{supp}\,F))$. The rest is same as in the Lemma 2.5 in \cite{OW05}.
\end{proof}

\subsection{Proof of Theorem \ref{Theorem:Lower_Bound_Open_Case}}

\begin{defn}[Hofer \cite{Ho90}]
    The \textit{displacement energy} of a subset $A\subset M$ with respect to the pseudo-distance $\rho$ is defined as
    \[e(A)=\inf\{\rho(\psi,\mathrm{Id})\mid \psi\in\mathrm{Ham}(M,\omega),\,\psi(A)\cap A=\emptyset\},\]
    if the above set is non-empty, and $e(A)=\infty$ otherwise.
\end{defn}

\begin{thm}[Theorem 1.3.A in \cite{EP93}]\label{Theorem:EP_displacement_positive}
    If $\rho$ is a genuine metric on $\mathrm{Ham}(M,\omega)$, then $e(U)>0$ for every non-empty open set $U\subset M$.
\end{thm}

\noindent This result allows us to reduce the proof of Theorem \ref{Theorem:Lower_Bound_Open_Case} to the following claim:

\begin{claim}[See Claim 4.3 in \cite{OW05}]\label{Claim:support_to_0_norm_to_0}
    If $F_i\in C^{\infty}_{0,c}(M)$ is a sequence of functions that satisfies $\sup\{\|F_i\|_{\infty}\}<\infty$ and $\mathrm{Vol}(\mathrm{supp}\,F_i)\xrightarrow{i\to\infty} 0$, then $\|F_i\|\xrightarrow{i\to\infty} 0$.
\end{claim}

\noindent Let $B \subset M$ be an embedded open ball with boundary $\partial B$ an embedded sphere, small enough to be displaced by the time-$1$ map of a Hamiltonian $H:[0,1]\times M \to \mathbb{R}$. Let $G:[0,1]\times M \to \mathbb{R}$ be obtained from $H$ by smoothly cutting off outside a neighbourhood $U_t$ of $\phi_H^t(\partial B)$. Then $\phi_G^1$ still displaces $B$, since $\phi_G^t(\partial B) = \phi_H^t(\partial B)$. By Claim~\ref{Claim:support_to_0_norm_to_0}, shrinking $U_t$ makes $\|G\|$ arbitrarily small. Hence the displacement energy of $B$ vanishes, and Theorem~\ref{Theorem:EP_displacement_positive} implies that $\rho$ is degenerate.

\begin{proof}[Proof of Claim \ref{Claim:support_to_0_norm_to_0}]
Let $\mathbf{1}_{U}$ denote the characteristic function of the set $U\subset M$. We prove
\begin{equation}\label{EquationOW:indicator}
    \|\mathbf{1}_U\|'\to 0\text{ as }\mathrm{Vol}(U)\to 0,
\end{equation}
Here, $\|\cdot\|'$ denotes the extension of $\|\cdot\|$ to $L^{\infty}_c(M)$ as in Theorem~\ref{TheoremOW:extension}. 
Since $\|\cdot\|$ is not bounded below by a positive multiple of the $L_{\infty}$-norm, for any $\varepsilon > 0$ there exists $F \in C^{\infty}_{0,c}(M)$ with $\|F\|_{\infty} = 1$ and $\|F\| = \|F\|' < \varepsilon$. 
Choose a small open set $U \subset M$ where $|F(x)| > 1 - \varepsilon$, and set $V := (\mathrm{supp}\,F) \setminus U$. 
Then, applying Lemma~\ref{LemmaOW:partition}, we obtain
\begin{equation}\label{EquationOW:bound1}
    \|\langle F\rangle_{U}\mathbf{1}_{U}\|'\leq\|\langle F\rangle_U\mathbf{1}_U+\langle F\rangle_V\mathbf{1}_V\|'+\|\langle F\rangle_V\mathbf{1}_V\|'\leq\|F\|'+\|\langle F\rangle_V\mathbf{1}_V\|'.
\end{equation}

\noindent From $\int_{M}F\,\omega^n=0$ we get $\mathrm{Vol}(U)\langle F\rangle_U+\mathrm{Vol}(V)\langle F\rangle_V=0$. Combining this with the fact that $\|\cdot\|\leq C(\|\cdot\|_{\infty}+\|\cdot\|_{L^1})$ we get
\[\|\langle F\rangle_V\mathbf{1}_V\|'=\Big\|\frac{\mathrm{Vol}(U)\langle F\rangle_U}{\mathrm{Vol}(V)}\mathbf{1}_V\Big\|'\leq\frac{\mathrm{Vol}(U)}{\mathrm{Vol}(V)}(\|\mathbf{1}_V\|_{\infty}+\|\mathbf{1}_V\|_{L^1})<\varepsilon,\]
provided $\mathrm{Vol}(U)$ is small enough. Now (\ref{EquationOW:bound1}) implies $\|\langle F\rangle_U\mathbf{1}_U\|'<\|F\|'+\varepsilon<2\varepsilon$. Using the fact that $|\langle F\rangle_U|>1-\varepsilon$, and taking $\varepsilon<1/2$ we get $\|\mathbf{1}_U\|'<4\varepsilon$. Since $\|\cdot\|'$ is invariant under compactly supported area preserving bijections, this applies to every bounded set $\widetilde{U}$ with the same measure as $U$, which completes the proof of (\ref{EquationOW:indicator}).

Let $F\in C^{\infty}_{c}(M)$. For $\varepsilon>0$ consider a finite partition $\mathrm{supp}\,F=\bigsqcup_{i=1}^{N}S_i$ into measurable sets $\{S_i\}_{i=1}^{N}$ with $\max( F|_{S_i})-\min(F|_{S_i})\leq\varepsilon$ for every $1\leq i\leq N$. We have
\begin{equation}\label{EquationOW:A}
    \|F\|'=\big\|\sum_{i=1}^N F\cdot\mathbf{1}_{S_i}\big\|'\leq \big\|\sum_{i=1}^{N}(F-F(\eta_i))\cdot\mathbf{1}_{S_i}\big\|'+\big\|\sum_{i=1}^{N}F(\eta_i)\cdot\mathbf{1}_{S_i}\big\|',
\end{equation}
where $\eta_i\in S_i$ is an arbitrary point. Assume that $F(\eta_i)\leq F(\eta_j)$ for $i\leq j$. Using the fact that $\|\cdot\|\leq C(\|\cdot\|_{\infty}+\|\cdot\|_{L^1})$ and the fact that $\|\sum_{i=1}^N(F-F(\eta_i))\cdot \mathbf{1}_{S_i}\|_{\infty}\leq\varepsilon$, we get
\begin{equation}\label{EquationOW:B}
    \|\sum_{i=1}^N(F-F(\eta_i))\cdot \mathbf{1}_{S_i}\|'\leq C\varepsilon(1+\mathrm{Vol}(\mathrm{supp}\,F)).
\end{equation}
Additionally, define $F(\eta_0)=0$ and use Abel's summation formula to get
\begin{equation}
\begin{aligned}\label{EquationOW:C}
\big\|\sum_{i=1}^N F(\eta_i)\cdot\mathbf{1}_{S_i}\big\|' 
&= \big\|\sum_{i=1}^{N}(F(\eta_i)-F(\eta_{i-1}))\cdot\mathbf{1}_{\cup_{k=i}^N S_k}\big\|' \\
&\leq \Big(\sum_{i=1}^{N} F(\eta_i)-F(\eta_{i-1})\Big)
   \cdot \max_{1\leq i\leq N}\|\mathbf{1}_{\cup_{k=i}^N S_k}\|' \\
&\leq \|F\|_{\infty}\cdot\max_{1\leq i\leq N}\|\mathbf{1}_{\cup_{k=i}^N S_k}\|'.
\end{aligned}
\end{equation}
Note that for every $1\leq i\leq N$ we have $\mathrm{Vol}(\cup_{k=i}^N S_k)\to 0$ as $\mathrm{Vol}(\mathrm{supp}\,F)\to 0$, so combining (\ref{EquationOW:indicator}), (\ref{EquationOW:A}),(\ref{EquationOW:B}) and (\ref{EquationOW:C}) we get the desired result.

\end{proof}

\end{document}